\documentclass[11pt, letterpaper]{article}
\usepackage[latin1]{inputenc}
\usepackage[T1]{fontenc}
\usepackage{amsmath,amssymb,amstext}
\usepackage{theorem}
\usepackage{multicol}
\usepackage[english]{babel}
\usepackage{enumerate}
\usepackage{eufrak}
\usepackage{graphicx}
\usepackage{caption}
\usepackage{subcaption}
\usepackage{float}

\graphicspath{ {./images/} }

\newtheorem{theorem}{Theorem}

\newtheorem{lemma}[theorem]{Lemma}

\newtheorem{proposition}[theorem]{Proposition}
\newtheorem{remark}[theorem]{Remark}

\newenvironment{proof}[1][Proof]{\textbf{#1.} }{\ \rule{0.5em}{0.5em}}

\renewcommand{\geq}{\geqslant}

\def\1{{\mathbf{1}}}

\def\1{{\mathbf{1}}}
\def\0.5{{\frac{1}{2}}}

\begin{document}

\title{AR(1) processes driven by second-chaos white noise: Berry-Esséen
bounds for quadratic variation and parameter estimation}
\author{\textbf{Soukaina Douissi}$^{1}$, \textbf{Khalifa Es-Sebaiy}$^{2}$,
\textbf{Fatimah Alshahrani}$^{3}$, \textbf{Frederi G. Viens}$^{4}$. \vspace{%
3mm} \\
{\normalsize {$^{1}$ Laboratory LIBMA, Faculty Semlalia, Cadi Ayyad
University, 40000 Marrakech, Morocco.}}\\
{\normalsize {Email: douissi.soukaina@gmail.com}}\\
{\normalsize {$^{2}$ Department of Mathematics, Faculty of Science, Kuwait
University, Kuwait.}}\\
{\normalsize {Email: khalifa.essebaiy@ku.edu.kw}}\\
{\normalsize{$^{3}$ Department of mathematical science, Princess Nourah bint Abdulrahman university, Riyadh.}}\\
{\normalsize {$^{3,4}$ Department of Statistics and Probability, Michigan
State University, East Lansing, MI 48824.}}\\
{\normalsize {Email: alshah10@msu.edu, viens@msu.edu}}}
\maketitle

{\ \noindent \textbf{Abstract:} In this paper, we study the asymptotic
behavior of the quadratic variation for the class of AR(1) processes driven
by white noise in the second Wiener chaos. Using tools from the analysis on
Wiener space, we give an upper bound for the total-variation speed of
convergence to the normal law, which we apply to study the estimation of the
model's mean-reversion. Simulations are performed to illustrate the
theoretical results.\vspace*{0.1in}}

{\noindent }\textbf{Key words}: {Central limit theorem; Berry-Esséen;
Malliavin calculus; parameter estimation; time series; Wiener chaos\vspace*{%
0.1in}}

 \footnotetext{%
The first author is supported by the Fulbright joint supervision program for
PhD students for the academic year 2018-2019 between Cadi Ayyad University
and Michigan State University. The fourth author is partially supported by
NSF awards DMS 1734183 and 1811779, and ONR award N00014-18-1-2192.}

\section{Introduction}

The topic of statistical inference for stochastic processes has a long
history, addressing a number of issues, though many difficult questions
remain. At the same time, a number of application fields are anxious to see
some practical progress in a selection of directions. Methodologies are
sought which are not just statistically sound, but stand a good chance of
being computationally implementable, if not nimble, to help practitioners
make data-based decisions in stochastic problems with complex time
evolutions. In this paper, which is motivated by parameter estimation within
the above context, we propose a quantitatively sharp analysis in this
direction, and we honor the scientific legacy of Prof. Larry Shepp.

Prof. Shepp is widely known for his seminal work on stochastic control,
optimal stopping, and applications in areas such as investment finance.
Often labeled as an applied probabilist, by those working in that area, he
had the merit, among many other qualities, of showing by example that
research activity in this area could benefit from an appreciation for the
mathematical aesthetics of constructing stochastic objects for their own
sake. His papers also showed that one's work is only as applied as one's
ability to calibrate a stochastic model to a realistic scenario. As obvious
as this view may seem, it is nonetheless in short supply in some current
circles, where model sophistication seems to replace all other imperatives.
Instead, we believe some of the principles guiding applied probability
research should include (i) statistical parsimony and robustness, (ii)
feature discovery, and above all, (iii) real-world impact where
mathematicians propose a real solution to a real problem. We think that
Prof. Shepp would not have been shy about agreeing that his seminal and
highly original works on optimal stopping and stochastic control \cite%
{SStopping}, \cite{SControl}, including the invention \cite{SRussian} of the
widely used Russian financial option, illustrate items (ii) and (iii) in
this philosophy perfectly. This leaves the question of how to estimate model
parameters needed to implement applied solutions. Prof. Shepp proved on many
occasions that this concern was also high on his list of objectives in
applied work; he proposed methods aligning with our stated principle (i)
above. The best example is the work for which Shepp is most widely known
outside of our stochastic circles: the mathematical foundation of the
Computational Tomography (CT) scanner, and in particular, the basis \cite%
{SCT} for its data analysis. Prof. Shepp is less well known for his direct
interest in statistically motivated stochastic modeling; the posthumous
paper \cite{EBSW} is an instance of this, on asymptotics of auto-regressive
processes with normal noise (innovations).

Our paper honors this legacy by providing a detailed and mathematically
rigorous stochastic analysis of some building blocks needed in the data
analysis of a simple class of stochastic processes. Our paper's originality
is in working out detailed quantitative properties for auto-regressive
processes with innovations in the second Wiener chaos. Our framework is
parsimonious in the sense of being determined by a small number of
parameters, while covering features of stationarity, mean-reversion, and
heavier-than-normal tail weight. We focus on establishing rates of
convergence in the central limit theorem for quadratic variations of these
processes, which we are then able to transfer to similar rates for the
model's moments-based parameter estimation. This precision would allow
practitioners to determine the validity and uncertainty quantification of
our estimates in the realistic setting of moderate sample size. Careless use
of a method of moments would ignore the potential for abusive conclusions in
this heavy-tailed time-series setting.\bigskip

The remainder of this introduction begins with an overview of the landscape
of parameter estimation for stochastic processes related to ours. The few
included references call for the reader to find additional references
therein, for the sake of conciseness. We then introduce the specific model
class used in this paper. It represents a continuation of the current
literature's motivation to calibrate stochastic models with features such as
stochastic memory and path roughness. It constitutes a departure from the
same literature's focus on the framework of Gaussian noise.

\subsection{Parameter estimation for stochastic processes: historical and
recent context}

Some of the early impetus in parameter estimation for stochastic processes
was inspired by classical ideas from frequentist statisitics, such as the
theoretical and practical superiority of maximum likelihood estimation
(MLE), over other, less constrained methodologies, in many contexts. We will
not delve into the description of many such instances, citing only the
seminal account \cite{LS}, first published in Russian in 1974 (see
references therein in Chapter 17, such as \cite{Nov71}). This was picked up
two decades later in the context of processes driven by fractional Brownian
motion, where it was shown that the martingale property used in earlier
treatments was not a necessary ingredient to establish the properties of
such MLEs: see in particular the treatment of processes with fractional
noise in \cite{KL} and in \cite{TV}. It was also noticed that least-squares
ideas, which led to MLEs in cases of white-noise driven processes, did not
share this property in the case of processes driven by fractional noise:
this was pointed out in the continuous-time based paper \cite{HN}. See also
a more detailed account of this direction of work in \cite{EEV} and
references therein, including a discussion of the distinction between
estimators based on continuous paths, and those using discrete sampling
under in-fill and increasing-horizon asymptotics. These were applied
particularly to various versions of the Ornstein-Uhlenbeck process, as
examples of processes with stationary increments and an ability to choose
other features such as path regularity and short or long memory.

The impracticality of computing MLEs for parameters of stochastic processes
in these feature-rich contexts, led the community to consider other
methodologies, looking more closely at least squares and beyond. A popular
approach is to work with incarnations of the method of moments. A full study
in the case of general stationary Gaussian sequences, with application to
in-fill asymtotics for the fractional Ornstein-Uhlenbeck process, is in \cite%
{BV}. This paper relates the relatively long history of those works where
estimation of a memory or Hölder-regularity parameter uses moments-based
objects, particularly quadratic variations. It also shows that the
generalized method of moments can, in principle, provide a number of options
to access vectors of parameters for discretely observed Gaussian processes
in a practical way. This was also illustrated recently in \cite{EV}, where
the Malliavin calculus and its connection to Stein's method was used to
establish speeds of convergence in the central-limit theorems for
quadratic-variations-based estimators for discretely observed processes. The
Stein-Malliavin technical methodology employed in \cite{EV} is that which
was introduced by Nourdin and Peccati in 2009, as described in their 2012
research monograph \cite{NP-book}.

Other estimation methods are also proposed for general stationary time
series, which we mention here, though they fall out of the scope of our
paper, and they do not lead to the same precision as those based on the
Stein-Malliavin method: see e.g. \cite{VVI} and \cite{VVI2} for the
Yule-Walker method and extensions. While the paper \cite{V} establishes that
essentially every continuous-time stationary process can be represented as
the solution of a Langevin (Ornstein-Uhlenbeck-type) equation with an
appropriate noise distribution, the two aforementioned follow-up papers \cite%
{VVI, VVI2}, which present an analog in discrete time, do not, however,
connect the discrete and continuous frameworks via any asymptotic theory.

Following an initial push in \cite{TV}, most of the recent papers mentioned
above, and recent references therein, state an explicit effort to work with
discretely observed processes. At least in the increasing-horizon case, the
papers \cite{EV} and \cite{DEV} had the merit of pointing out that many of
the discretization techniques used to pass from continuous-path to
discrete-observation based estimators, were inefficient, and it is
preferable to work directly from the statistics of the discretely observed
process. Our paper picks up this thread, and introduces a new direction of
research which, to our knowledge, has not been approached by any authors:
can the asymptotic normality of quadratic variations and related estimators,
including very precise results on speeds of convergence, be obtained when
the driving noise is not Gaussian?

The main underlying theoretical result we draw on is the optimal estimation
of total-variation distances between chaos variables and the normal law,
established in \cite{NP2015}. It was used for quadratic variations of
stationary sequences in the Gaussian case in \cite{NV2014}. But when the
Gaussian setting is abandonned, the result in \cite{NP2015} cannot be used
directly. Instead, our paper makes a theoretical advance in the analysis on
Wiener space, by drawing on a simple idea in the recent preprints \cite{NZ}
and \cite{NPY}; our main result provides an example of a sum of chaos
variables whose distance to the normal appears to be estimated optimally,
whereas a standard use of the Schwartz inequality would result in a much
weaker result. The precise location of the technique leading to this
improvement is pointed out in the main body of our paper: see Theorem \ref%
{newbound}, particularly inequality (\ref{newboundineq}) in its proof and
the following brief discussion there, and Remark \ref{KeyRem}. This allows
us to prove our Berry-Esséen-type speed of $n^{-1/2}$, rather than what
would have resulted in a speed of $n^{-1/4}.$

\subsection{A stationary process with second-chaos noise, and related
literature}

Given our intent to address the new issue of noise distribution, and knowing
that Berry-Esséen-type questions for models with mere Gaussian noise already
present technical challenges, we choose to minimize the number of technical
issues to address in this paper by focusing on the simplest possible
stationary model class which does not restrict the marginal noise
distribution within a family which is tractable using Wiener chaos analysis
and tools from the Malliavin calculus. This is the auto-regressive model of
order 1 (a.k.a. AR(1)) with independent noise terms, where the noise
distribution is in the second Wiener chaos, i.e.
\begin{equation}
Y_{n}=a_{0}+a_{1}Y_{n-1}+\varepsilon _{n}  \label{Modelintro}
\end{equation}%
where $\left\{ \varepsilon _{n};n\in \mathbf{Z}\right\} $ is an i.i.d.
sequence in the second Wiener chaos, and $a_{0}$ and $a_{1}$ are constants.
The complete description and construction of this process and of the noise
sequence is given in Section \ref{Presentation}, see (\ref{ARmodel}).

As explained in Section \ref{Prelim}, the second Wiener chaos is a linear
space, and since the model (\ref{Modelintro}) is linear, its solution, if
any, lies in the same chaos. This points to a simple theoretical motivation
and justification for studying the increasing horizon problem as opposed to
the in-fill problem. We also include a practical motivation for doing so,
further below in this section, coming from an environmental statistics
question.

For the former motivation, note that the AR(1) specification (\ref%
{Modelintro}), with essentially any square-integrable i.i.d. noise
distribution, is known to converge weakly, after appropriate aggregation and
scaling, to the so-called Ornstein-Uhlenbeck process (also known
occasionally as the Vasicek process), which solves the stochastic
differential equation
\begin{equation}
dX_{t}=\alpha (m-X_{t})dt+\sigma dW\left( t\right)   \label{OU}
\end{equation}%
where $W$ is a standard Wiener process (Gaussian Brownian motion), and the
parameters $\alpha ,m$, and $\sigma $ are explicitly related to $a_{0},a_{1}$
and $Var\left[ \varepsilon _{n}\right] $. See for instance \cite[Chapter 2]%
{T}, which covers the case of all square-integrable innovations; this paper
assumes a piecewise linear interpolation in the normalization, which could
be eliminated by switching to convergence in the Skorohod $J_{1}$ topology.
A reference avoiding linear interpolation, with convergence in the Skorohod $%
J_{1}$ topology, is \cite{CS}, where innovations are assumed to have four
moments. In any case, this central limit theorem constrains the modeling of
stationary/ergodic processes via diffusive differential formulation: under
in-fill asymptotics with weakly dependent noise, the AR(1) specification
cannot preserve any non-normal noise distribution in the limit. It is of
course possible to interpolate the above process $Y$ in a number of ways, to
result in a continuous-time process whose discrete-time marginals are those
specified via (\ref{Modelintro}).

However we believe it is difficult or impossible to give a linear
diffusion-type stochastic differential equation, akin to (\ref{OU}), whose
fixed-time-step marginals are as in (\ref{Modelintro}) for an arbitrary
noise distribution, while simultaneously describing what second-chaos
process differential would need to replace $W$ in (\ref{OU}). The so-called
Rosenblatt process (see \cite{TRosenblatt}), the only known second-chaos
continuous-time process with a stochastic calculus similar to $W$'s, gives
an example of a viable alternative to (\ref{OU}) living in the second chaos.
But this process is known to have only a long-memory version. Thus it cannot
be a proxy for any continuous-time analogue of (\ref{Modelintro}), since the
noise there has no memory. Similar issues would presumably exist for other
non-Gaussian AR(1) and related auto-regressive processes. A few have been
studied recently in the literature. We mention \cite{H, YG}, which cover
various noise structures similar to second-chaos noises, and here again, no
asymptotic or interpolation theory is provided to relate to continuous time.
There does exist a general treatment in \cite{EBSW} of asymptotics for all
AR($p$) processes: the limit processes are the so-called Continuous-AR($p$)
processes, which are Gaussian, and have $p-1$-differentiable paths (a form
of very long memory for $p>1$); that paper assumes normal innovations to
keep technicalities to a minimum.

Another indication that finding such a proxy may fail comes in the specific
case of the so-called Gumbel distribution for $\varepsilon _{n}$. This law
is a popular distribution for extreme-value modeling. The fact that this law
is in the second chaos is a classical result (as a weighted sum of
exponentials, see \cite{Sch51}), though it does not appear to be widely know
in the extreme-value community. The standard (mean-zero) Gumbel law can be
represented as $\sum_{n=1}^{\infty }j^{-1}\left( E_{j}-1\right) $ where $%
E_{j}=\left( N_{j}^{2}+\bar{N}_{j}^{2}\right) /2$ is a standard exponential
variable (chi-squared with two degrees of freedom, $N_{j}$ and $\bar{N}_{j}$
are iid standard normals). The Gumbel law is known to give rise to a
second-chaos version of an isonormal Gaussian process, known as the Gumbel
noise measure (or Gumbel process); that stochastic measure obeys the same
laws as the white-noise measure (including independence of increments which
fails for the Rosenblatt noise), if one replaces the standard algebra of the
reals by the max-plus algebra. This is explained in detail in the preprint
\cite{MTM}; also see references therein. By virtue of this change of
algebra, stochastic differential specifications as in (\ref{OU}) cannot be
defined using the Gumbel noise.

However, the discrete version of the Gumbel noise, an i.i.d. sequence $%
\left( \varepsilon _{n}\right) _{n}$ with Gumbel marginals, is a good
example of a noise type which can be used in the AR(1) process (\ref%
{Modelintro}). This specific model, known as the AR(1) process with Gumbel
noise (or innovations), presents a main motivation for our work. Recent
references on this process, and on the closely related process where the
marginals of $Y$ are Gumbel-distributed, include \cite{GumbelBayes} for a
Bayesian study, \cite{GumbelMarginal} for applications to maxima in
greenhouse gas concentration data, and \cite{ExtremeAR} for AR processs in
the broader extreme-value context. A survey on AR(1) models with different
types of innovations and marginals, while not including the Gumbel, is in
the unpublished manuscript \cite{AR1Survey}. The use of the Gumbel
distribution for describing environmental time series, mainly when looking
at extremes, is fairly widespread, but we do not cite this literature
because it does not appear willing to acknowledge that time-series models
driven by Gumbel innovations should be used, rather than using tools for
i.i.d. Gumbel data. This literature, which is easy to find, is also entirely
unaware that the Gumbel distribution is in the second Wiener chaos.

All these reasons give us ample cause to investigate the basic
method-of-moments building blocks for determining parameters in stationary
time series with second-chaos innovations. For the sake of concentrating on
the core mathematical analysis towards this end, we focus on the asymptotics
of quadratic variations for models in the class (\ref{Modelintro}). The
methodology developed in \cite{EV} can then be adapted to handle any
method-of-moments-based estimators, at the cost of some additional effort.
We provide examples of this in the latter sections of this paper. Our main
result is that, for any second-chaos innovations in (\ref{Modelintro}), the
quadratic variation of $Y$ has explicit normal asymptotics, with a speed of
convergence in total variation which matches the classical Berry-Esseén
speed of $n^{-1/2}$.

~

The remainder of this paper is structured as follows. Section \ref{Prelim}
provides elements from the analysis on Wiener space which will be used in
the paper. Section \ref{Presentation} presents the details of the class of
AR(1) models we will analyze. Section \ref{Asymp} computes the asymptotic
variance of the AR(1)'s quadratic variation by looking separately at its
2nd-chaos and 4th-chaos components, whose asymptotics are of the same order.
Section \ref{BerryEsseen} establishes our main result, the Berry-Esseén
speed of convergence in total-variation for the normal fluctuations of the
AR(1)'s quadratic variation. Finally, Section \ref{Estim} defines a
method-of-moments estimator for the mean-reversion rate of this AR(1)
process, and establishes its asymptotic properties; a numerical study is
included to gauge the distance between this renormalized estimator and the
normal law.

\section{Preliminaries}

\label{Prelim}

In this first section, we recall some elements from stochastic analysis that
we will need in the paper. See \cite{NP-book}, \cite{nualart-book}, and \cite%
{ustu} for details. Any real, separable Hilbert space ${\mathcal{H}}$ gives
rise to an isonormal Gaussian process: a centered Gaussian family $%
(G(\varphi ),\varphi \in {\mathcal{H}})$ of random variables on a
probability space $(\Omega ,\mathcal{F},\mathbf{P})$ such that $\mathbf{E}%
(G(\varphi )G(\psi ))=\langle \varphi ,\psi \rangle _{{\mathcal{H}}}$. In
this paper, it is enough to use the classical Wiener space, where $\mathcal{H%
}=L^{2}([0,1])$, though any $\mathcal{H}$ will also work. In the case ${%
\mathcal{H}}=L^{2}([0,1])$, $G$ can be identified with the stochastic
differential of a Wiener process $W$ and one interprets $G(\varphi
):=\int_{0}^{1}\varphi \left( s\right) dW\left( s\right) $.

The Wiener chaos of order~$n$ is defined as the closure in $L^{2}\left(
\Omega \right) $ of the linear span of the random variables $H_{n}(G(\varphi
))$, where $\varphi \in {\mathcal{H}},\Vert \varphi \Vert _{{\mathcal{H}}}=1$
and $H_{n}$ is the Hermite polynomial of degree~$n$. The intuitive
Riemann-sum-based notion of multiple Wiener stochastic integral $I_{n}$ with
respect to $G\equiv W$, in the sense of limits in $L^{2}\left( \Omega
\right) $, turns out to be an isometry between the Hilbert space ${\mathcal{H%
}}^{\odot n}$ (symmetric tensor product) equipped with the scaled norm $%
\frac{1}{\sqrt{n!}}\Vert \cdot \Vert _{{\mathcal{H}}^{\otimes n}}$ and the
Wiener chaos of order~$n$ under $L^{2}\left( \Omega \right) $'s norm. In any
case, we have the following fundamental decomposition of $L^{2}\left( \Omega
\right) $ as a direct sum of all Wiener chaos.

\noindent $\bullet $ \textbf{The Wiener chaos expansion.} For any $F\in
L^{2}\left( \Omega \right) $, there exists a unique sequence of functions $%
f_{n}\in {\mathcal{H}}^{\odot n}$ such that
\begin{equation*}
F=\mathbf{E}[F]+\sum_{n=1}^{\infty }I_{n}(f_{n}),
\end{equation*}%
where the terms are all mutually orthogonal in $L^{2}\left( \Omega \right) $
and
\begin{equation}
\mathbf{E}\left[ I_{n}(f_{n})^{2}\right] =n!\Vert f_{n}\Vert _{{\mathcal{H}}%
^{\otimes n}}^{2}.  \label{Isometryprop}
\end{equation}

\noindent $\bullet $ \textbf{Product formula and contractions.} Since $%
L^{2}\left( \Omega \right) $ is closed under multiplication, the special
case of the above expansion exists for calculating products of Wiener
integrals, and is explicit using contractions: for any $p$, $q$, and
symmetric integrands $f\in \mathcal{H}^{\odot p}$ and $g\in \mathcal{H}%
^{\odot q}$,
\begin{equation}
I_{p}(f)I_{q}(g)=\sum_{r=0}^{p\wedge q}r!{C}_{p}^{r}{C}_{q}^{r}I_{p+q-2r}(f%
\otimes _{r}g);  \label{product}
\end{equation}%
see \cite[Proposition 1.1.3]{nualart-book} for instance; the contraction $%
f\otimes _{r}g$ is the element of ${\mathcal{H}}^{\otimes (p+q-2r)}$ defined
by
\begin{eqnarray*}
&&(f\otimes _{r}g)(s_{1},\ldots ,s_{p-r},t_{1},\ldots ,t_{q-r}) \\
&&\qquad :=\int_{[0,1]^{p+q-2r}}f(s_{1},\ldots ,s_{p-r},u_{1},\ldots
,u_{r})g(t_{1},\ldots ,t_{q-r},u_{1},\ldots ,u_{r})\,du_{1}\cdots du_{r}.
\end{eqnarray*}%
The special case for $p=q=1$ is particularly handy, and can be written in
its symmetrized form:%
\begin{equation*}
I_{1}(f)I_{1}(g)=2^{-1}I_{2}\left( f\otimes g+g\otimes f\right) +\langle
f,g\rangle _{{\mathcal{H}}}.
\end{equation*}

\noindent $\bullet $ \textbf{Hypercontractivity in Wiener chaos}. For $h\in {%
\mathcal{H}}^{\otimes q}$, the multiple Wiener integrals $I_{q}(h)$, which
exhaust the set ${\mathcal{H}}_{q}$, satisfy a hypercontractivity property
(equivalence in ${\mathcal{H}}_{q}$ of all $L^{p}$ norms for all $p\geq 2$),
which implies that for any $F\in \oplus _{l=1}^{q}{\mathcal{H}}_{l}$ (i.e.
in a fixed sum of Wiener chaoses), we have
\begin{equation}
\left( E\big[|F|^{p}\big]\right) ^{1/p}\leqslant c_{p,q}\left( E\big[|F|^{2}%
\big]\right) ^{1/2}\ \mbox{ for any }p\geq 2.  \label{hypercontractivity}
\end{equation}%
It should be noted that the constants $c_{p,q}$ above are known with some
precision when $F\in {\mathcal{H}}_{q}$: by Corollary 2.8.14 in \cite%
{NP-book}, $c_{p,q}=\left( p-1\right) ^{q/2}$.

\noindent $\bullet $ \textbf{Malliavin derivative and other operators on
Wiener space}. The Malliavin derivative operator $D$, and other operators on
Wiener space, are needed briefly in this paper, to provide an efficient
proof of the first theorem in Section \ref{BerryEsseen}, and to interpret an
observation of I. Nourdin and G. Peccati, given below in (\ref{NPobschaos}),
for a bound on the total variation distance of any chaos law to the normal
law. We do not provide any background on these operators, referring instead
to Chapter 2 in \cite{NP-book}, and briefly mentioning here the facts we
will use in the proof of Section \ref{BerryEsseen}, without spelling out all
assumptions. Strictly speaking, all the results in this paper can be
obtained without the following facts, but this would be exceedingly tedious
and wholly nontransparent.

\begin{itemize}
\item The operator $D$ maps $I_{q}\left( f\right) $ to $t\mapsto
qI_{q-1}\left( f\left( .,t\right) \right) $ and is consistent with the
ordinary chain rule. Its domain is denoted by $\mathbf{D}^{1,2},$and
includes all chaos variables.

\item The operator $L$, known as the generator of the Orstein-Uhlenbeck
semigroup on Wiener space, maps $I_{q}\left( f\right) $ to $-qI_{q}\left(
f\right) $, and $L^{-1}$ denotes its pseudo-inverse: $L$'s kernel is the
constants, all other chaos are its eigenspaces. Combining this with the
previous point, we obtain $-D_{t}L^{-1}I_{q}\left( f\right) =I_{q-1}\left(
f\left( .,t\right) \right) $.

\item This $D$ has an adjoint $\delta $ in $L^{2}\left( \Omega \right) $,
which by definition satisfies the duality relation $\mathbf{E}\left\langle
DF,u\right\rangle _{\mathcal{H}}=\mathbf{E}\left[ F\delta \left( u\right) %
\right] $, where $u$ is any stochastic process for which the expressions are
defined. The domain of $\delta $ is a non-trivial object of study, but it is
known to contain all square-integrable $W$-adapted processes for the case of
$G=W$, the wiener process, where $\mathcal{H}=L^{2}\left( [0,1]\right) $.

\item We have the relation%
\begin{equation*}
F=\delta (-DL^{-1})F.
\end{equation*}
\end{itemize}

\noindent $\bullet $ \textbf{Distances between random variables}. The
following is classical. If $X,Y$ are two real-valued random variables, then
the total variation distance between the law of $X$ and the law of $Y$ is
given by
\begin{equation*}
d_{TV}\left( X,Y\right) :=\sup_{A\in \mathcal{B}({\mathbb{R}})}\left\vert P%
\left[ X\in A\right] -P\left[ Y\in A\right] \right\vert
\end{equation*}%
where the supremum is over all Borel sets. The Kolmogorov distance $%
d_{Kol}\left( X,Y\right) $ is the same as $d_{TV}$ except one take the sup
over $A$ of the form $(-\infty ,z]$ for all real $z$. The Wasserstein
distance uses Lipschitz rather than indicator functions:%
\begin{equation*}
d_{W}\left( X,Y\right) :=\sup_{f\in Lip(1)}\left\vert Ef(X)-Ef(Y)\right\vert
,
\end{equation*}%
$Lip(1)$ being the set of all Lipschitz functions with Lipschitz constant $%
\leqslant 1$.

\noindent $\bullet $ \textbf{The observation of Nourdin and Peccati}. Let $N$
denote the standard normal law. The following observation relates an
integration-by-parts formula on Wiener space with a classical result of Ch.
Stein.

\begin{description}
\item Let $X\in \mathbf{D}^{1,2}$ with $\mathbf{E}\left[ X\right] =0$ and $%
Var\left[ X\right] =1$. Then (see \cite[Proposition 2.4]{NP2015}, or \cite[%
Theorem 5.1.3]{NP-book}), for $f\in C_{b}^{1}\left( \mathbf{R}\right) $,
\begin{equation*}
E\left[ Xf\left( X\right) \right] =E\left[ f^{\prime }\left( X\right)
\left\langle DX,-DL^{-1}X\right\rangle _{\mathcal{H}}\right]
\end{equation*}%
and by combining this with properties of solutions of Stein's equations, one
gets
\begin{equation}
d_{TV}\left( X,N\right) \leqslant 2E\left\vert 1-\left\langle
DX,-DL^{-1}X\right\rangle _{\mathcal{H}}\right\vert .  \label{NPobs}
\end{equation}%
One notes in particular that when $X\in {\mathcal{H}}_{q}$, since $%
-L^{-1}X=q^{-1}X$, we obtain conveniently
\begin{equation}
d_{TV}\left( X,N\right) \leqslant 2E\left\vert 1-q^{-1}\left\Vert
DX\right\Vert _{\mathcal{H}}^{2}\right\vert .  \label{NPobschaos}
\end{equation}
\end{description}

\noindent $\bullet$ \textbf{A convenient lemma\label{CLemma}}. The following
result is a direct consequence of the Borel-Cantelli Lemma (the proof is
elementary; see e.g. \cite{KN}). It is convenient for establishing
almost-sure convergences from $L^{p}$ convergences.

\begin{lemma}
\label{Borel-Cantelli} Let $\gamma >0$. Let $(Z_{n})_{n\in \mathbb{N}}$ be a
sequence of random variables. If for every $p\geq 1$ there exists a constant
$c_{p}>0$ such that for all $n\in \mathbb{N}$,
\begin{equation*}
\Vert Z_{n}\Vert _{L^{p}(\Omega )}\leqslant c_{p}\cdot n^{-\gamma },
\end{equation*}
then for all $\varepsilon >0$ there exists a random variable $\eta
_{\varepsilon }$ which is almost such that
\begin{equation*}
|Z_{n}|\leqslant \eta _{\varepsilon }\cdot n^{-\gamma +\varepsilon }\quad %
\mbox{almost surely}
\end{equation*}
for all $n\in \mathbb{N}$. Moreover, $E|\eta _{\varepsilon }|^{p}<\infty $
for all $p\geq 1$.
\end{lemma}

\section{The model\label{Presentation}}

\subsection{Definition}

We consider the following AR(1) model
\begin{equation}
\left\{
\begin{array}{ll}
Y_{n}=a_{0}+a_{1}Y_{n-1}+\varepsilon _{n}, & n\geq 1 \\
\varepsilon _{n}=\sum\limits_{\delta =1}^{\infty }\sigma _{\delta
}(Z_{n,\delta }^{2}-1) &  \\
Y_{0}=y_{0}\in \mathbb{R}. &
\end{array}%
\right.  \label{ARmodel}
\end{equation}%
where $a_{0}$, $a_{1}$ and $\left\{ \sigma _{\delta },\delta \geq 1\right\} $
are real constants. The sequence of innovations $\left\{ \varepsilon
_{n},n\geq 1\right\} $ is i.i.d., with distribution in the second Wiener
chaos. It turns out that this sequence can be represented as in the second
line above in (\ref{ARmodel}), where the family $\left\{ Z_{n,\delta },n\geq
1,\delta \geq 1\right\} $ are i.i.d. standard Gaussian random variables
defined on $(\Omega ,\mathcal{F},\mathbf{P})$, and $\left\{ \sigma _{\delta
};\delta \geq 1\right\} $ is a sequence of reals satisfying
\begin{equation}
\sum\limits_{\delta =1}^{\infty }\sigma _{\delta }^{2}<\infty .
\label{assumption}
\end{equation}%
This is explained in \cite[Section 2.7.4]{NP-book}. We assume that the mean
reversion parameter $a_{1}$ is such that $\left\vert a_{1}\right\vert <1$.
Under this condition, (\ref{ARmodel}) also admits a stationary ergodic
solution. Both the version above and the stationary version are linear
functionals of elements of the form of $\varepsilon _{n}$, which are
elements of the second Wiener chaos. Since this chaos is a vector space,
both versions of $Y$ take values in the second Wiener chaos.

By truncating the series in (\ref{ARmodel}), one obtains a process which is
a sum of chi-squared variables, converging to $Y$ in $L^{2}(\Omega )$.
Special cases where the sum is finite, can be considered. In the figures
below, we simulate 500 observations from such cases, to show the variety of
behaviors, even with a limited number of terms in the noise series.

\begin{itemize}
\item[\textbullet ] When $\sigma _{1}=\sigma $ and $\sigma _{\delta }=0$,
for all $\delta \geq 2$, corresponds to a scaled mean-zero chi-squared white
noise with one degree of freedom: $(Z_{n,1}^{2}-1)\sim \chi ^{2}(1)$.

\item[\textbullet ] When $\sigma _{1}=\sigma _{2}=\sigma $ and $\sigma
_{\delta }=0$, $\forall \delta \geq 3$, an exponential white noise with rate
parameter $1/(2\sigma )$. Indeed $(Z_{n,1}^{2}-1)+(Z_{n,2}^{2}-1)\sim
\mathcal{E}(1/2)$.

\item[\textbullet ] When $\sigma _{1}=-\sigma _{2}$, and $\sigma _{\delta
}=0 $, for all $\delta \geq 2$, which is a symmetric second chaos white
noise, $\varepsilon $'s law is equal to a product normal law: if $%
N,N^{\prime }$ are two i.i.d. standard normals, then $2NN^{\prime }\sim
(Z_{n,1}^{2}-1)-(Z_{n,2}^{2}-1)=Z_{n,1}^{2}-Z_{n,2}^{2}$.
\end{itemize}

\begin{figure}[]
\begin{subfigure}[b]{0.5\linewidth}
    \centering
    \includegraphics[width=1\linewidth]{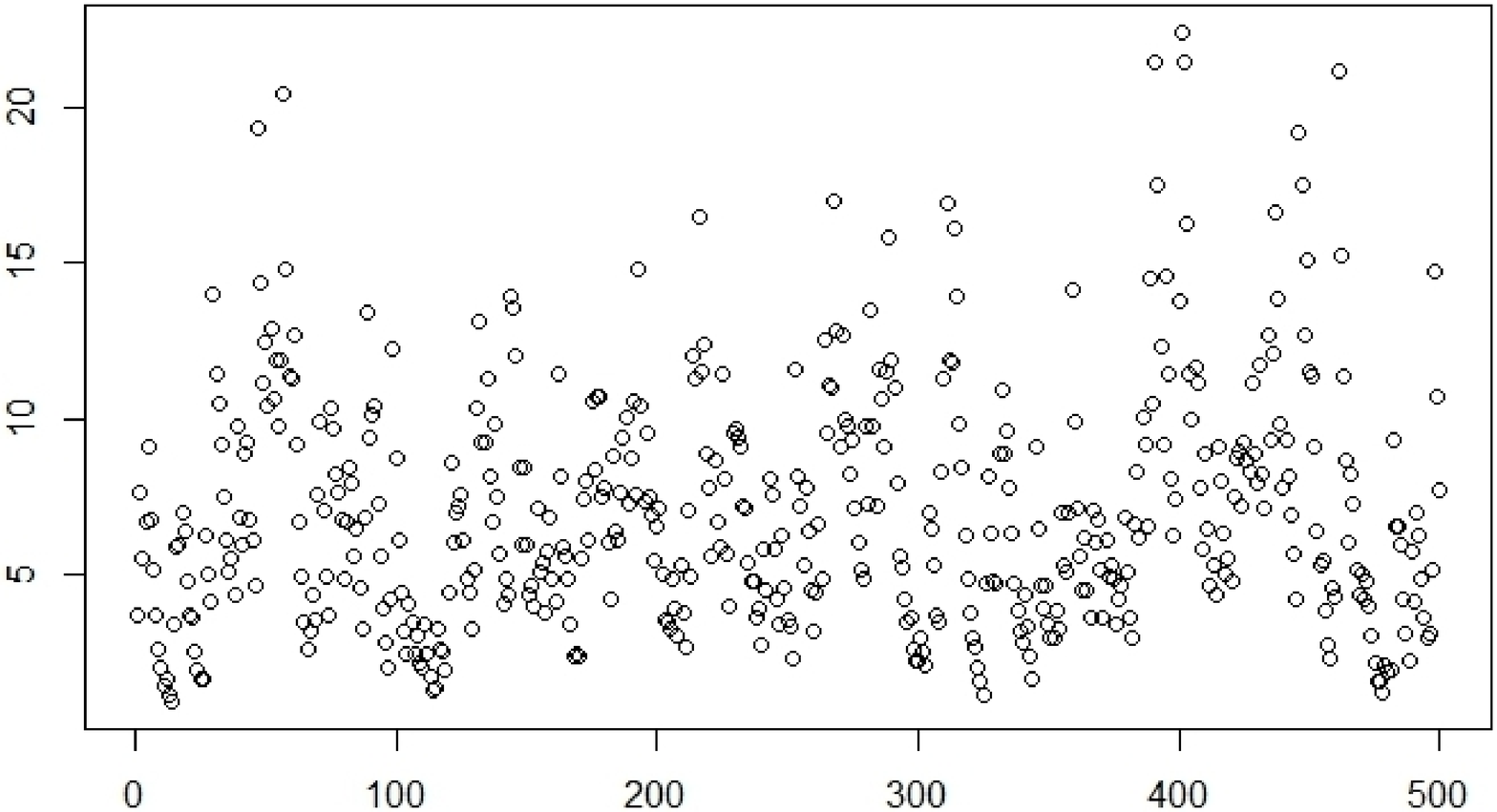}
   \caption{$a_{0}= 2$, $a_{1} =0.7$, $y_{0} = 3$, $\sigma = 2$.}
    \label{fig7:b}
\end{subfigure}
\begin{subfigure}[b]{0.5\linewidth}
    \centering
    \includegraphics[width=1\linewidth]{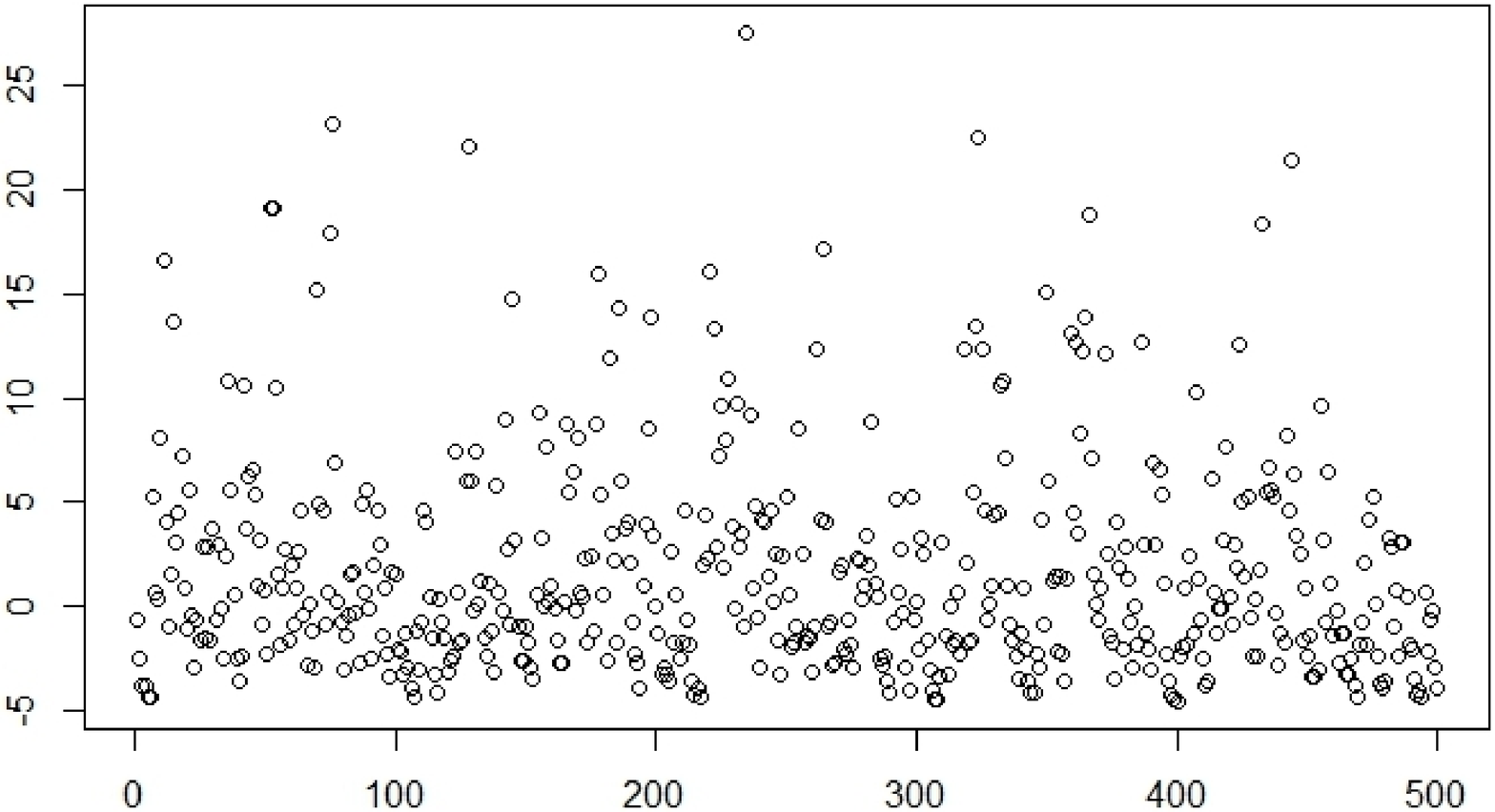}
    \caption{$a_{0}=1.2$, $a_{1} =0.4$, $y_{0} = 2$, $\sigma = 4$.}
    \label{fig7:d}
\end{subfigure}
\begin{subfigure}[b]{0.5\linewidth}
    \centering
    \includegraphics[width=1\linewidth]{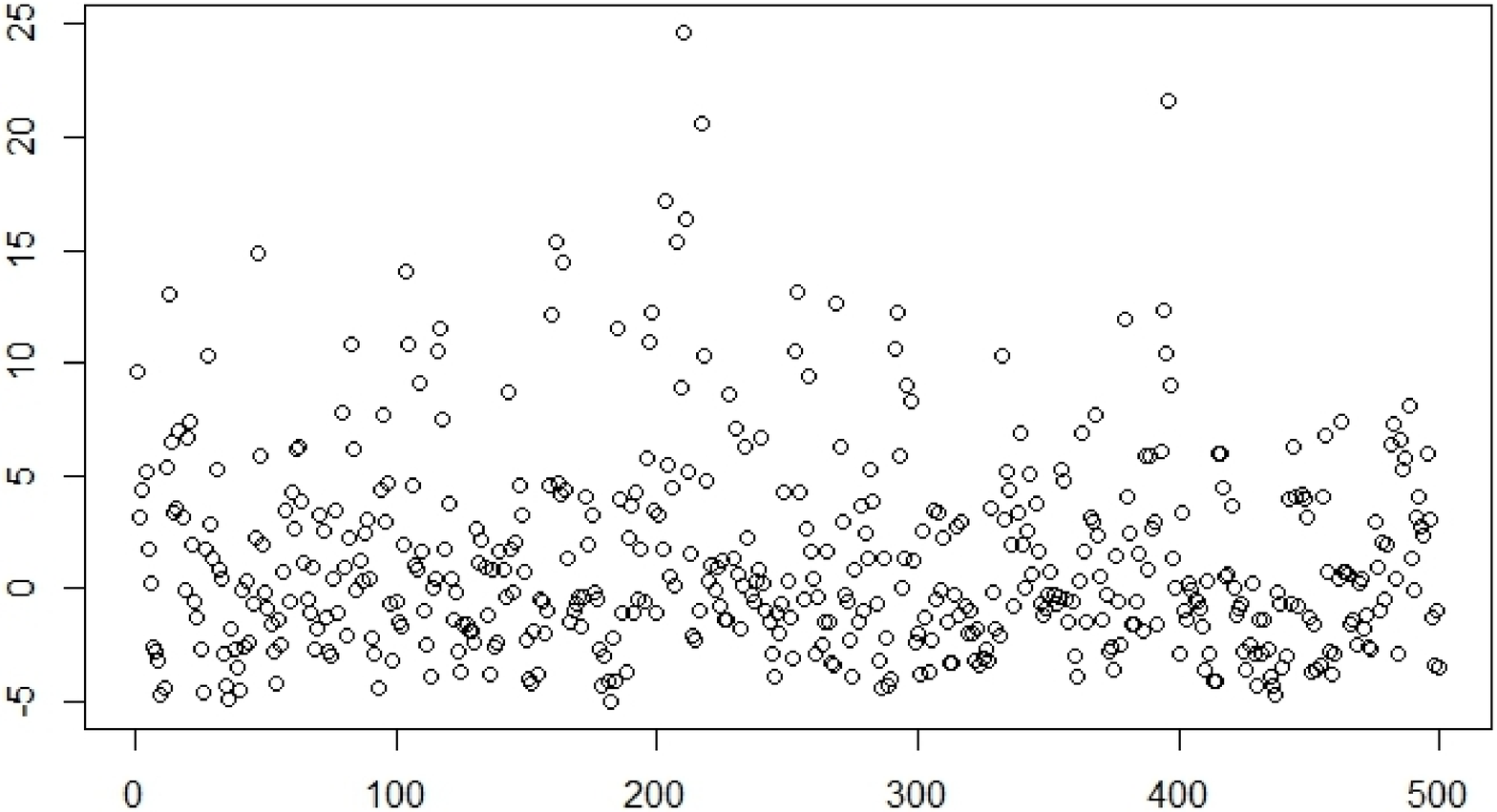}
    \caption{$a_{0}= 0.8$, $a_{1} =0.5$, $y_{0} = 1$, $\sigma_1 = \sigma_2= 2$.}
    \label{fig7:c}
\end{subfigure}
\begin{subfigure}[b]{0.5\linewidth}
        \centering
        \includegraphics[width=1\linewidth]{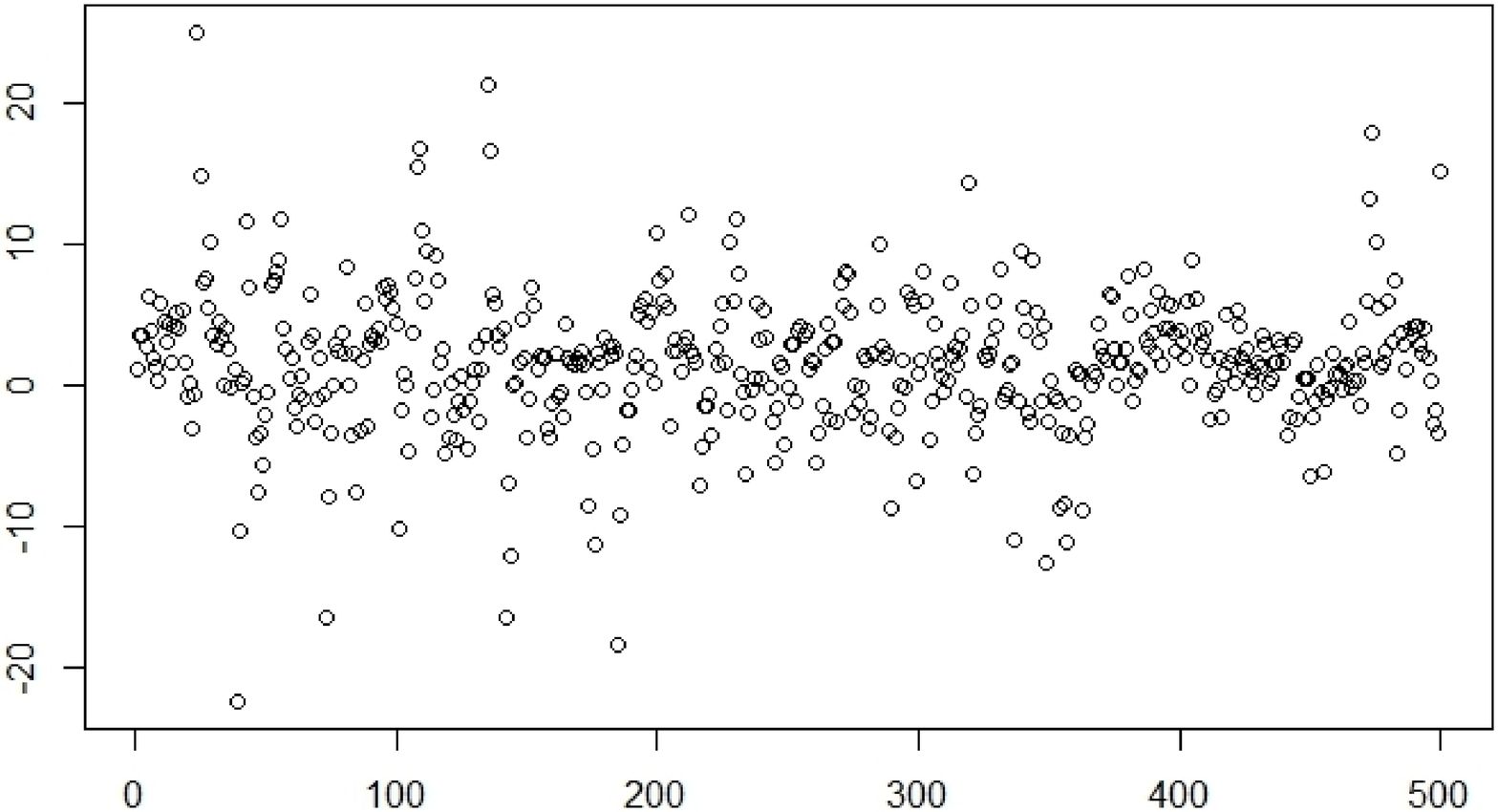}
        \caption{$a_{0}= 1.2$, $a_{1} =0.4$, $y_{0} = 0$, $\sigma_1 = - \sigma_2= 1$.}
        \label{fig7:c}
\end{subfigure}
\caption{ 500 various observations from (\protect\ref{ARmodel}) for
different values of $a_{0}$, $a_{1}$, $y_{0}$ and $\protect\sigma_{\protect%
\delta}$, $\protect\delta =1,2$. }
\label{fig7}
\end{figure}

\begin{remark}
We can see from the figures above the asymmetry in figures (a), (b) and (c)
due to the asymmetric nature of the noise; figure (d) shows more symmetry
because of the choice $\sigma _{1}=-\sigma _{2}$. We also notice that when
the mean reversion is fairly strong and the noise is large the shape of the
observations is balanced (figure (a)), while when the noise is larger
compared to the mean-reversion parameter, the observations look like an
Ornstein-Uhlenbeck process with a noise larger than the drift (see figure
(b)).
\end{remark}

\subsection{Quadratic variation}

This paper's main goal is to determine the asymptotic distribution of the
quadratic variation of the observations $\left\{ Y_{n},n\geq 1\right\} $
using analysis on Wiener space.

This will be facilitated by the fact, mentioned above, that the sequence $%
(Y_{n})_{n\geq 1}$ lives in the second Wiener chaos with respect to the
Wiener process $W$, by virtue of being the solution of a linear equation
with noise in the second chaos. To be more specific, observe that $\left\{
Y_{n},n\geq 1\right\} $ in (\ref{ARmodel}) can be expressed recursively as
follows :
\begin{equation}
Y_{i}=d_{i}+\sum_{k=1}^{i}a_{1}^{i-k}\sum\limits_{\delta =1}^{\infty }\sigma
_{\delta }\left( Z_{k,\delta }^{2}-1\right) ,\hspace{0.5cm}i\geq 1
\end{equation}%
where
\begin{equation}
d_{i}=a_{1}^{i}y_{0}+a_{0}\sum_{k=1}^{i}a_{1}^{i-k}.
\end{equation}%
For the sake of ease of computation in Wiener chaos, it will be convenient
throughout this paper to refer to the Wiener integral representation of the
noise terms $Z_{k,\delta }^{2}$. For this, there exists $\left\{ h_{k,\delta
},k\geq 1,\delta \geq 1\right\} $ an orthonormal family $L^{2}([0,1])$ for
which $Z_{k,\delta }=W(h_{k,\delta })=I_{1}^{W}\left( h_{k,\delta } \right)
=\int_{0}^{1}h_{k,\delta }\left( r\right) dW\left( r\right) $. Hence, using
the fact, which comes from the most elementary application of the product
formula (\ref{product}), that $W^{2}\left( \varphi \right)
-1=I_{2}^{W}\left( \varphi ^{\otimes 2}\right) $, we have for $i\geq 1$ :%
\begin{align*}
Y_{i}& =d_{i}+\sum_{k=1}^{i}a_{1}^{i-k}\sum_{\delta =1}^{\infty }\sigma
_{\delta }\left( Z_{k,\delta }^{2}-1\right) \\
& =d_{i}+\sum_{k=1}^{i}a_{1}^{i-k}\sum_{\delta =1}^{\infty }\sigma _{\delta
}\left( W^{2}(h_{k,\delta })-1\right) \\
& =d_{i}+\sum_{k=1}^{i}a_{1}^{i-k}\sum_{\delta =1}^{\infty }\sigma _{\delta
}I_{2}^{W}(h_{k,\delta }^{\otimes 2})
\end{align*}%
Therefore, using the linearity property of multiple integrals, we can write
for $i\geq 1$,
\begin{equation*}
Y_{i}=d_{i}+\tilde{Y}_{i},
\end{equation*}%
where
\begin{equation}
\tilde{Y}_{i}:=I_{2}^{W}(f_{i})\hspace{0.4cm}\text{and}\hspace{0.4cm}%
f_{i}:=\sum_{k=1}^{i}a_{1}^{i-k}\sum_{\delta =1}^{\infty }\sigma _{\delta
}h_{k,\delta }^{\otimes 2}.  \label{Y_{i}f_{i}}
\end{equation}%
A straightforward computation shows that under Assumption (\ref{assumption}%
), the kernel $f_{i}\in L^{2}([0,1]^{2})$ for all $i\geq 1$ : indeed
\begin{equation*}
\Vert f_{i}\Vert _{L^{2}([0,1]^{2})}^{2}=\sum_{\delta =1}^{\infty }\sigma
_{\delta }^{2}\times \frac{(1-a_{1}^{2i})}{(1-a_{1}^{2})}\leqslant \frac{1}{
(1-a_{1}^{2})}\sum_{\delta =1}^{\infty }\sigma _{\delta }^{2} <\infty .
\end{equation*}

Our main object of study in the next two sections is the asymptotics of the
quadratic variation defined as follows :%
\begin{equation}
Q_{n}:=\frac{1}{n}\sum_{i=1}^{n}(Y_{i}-d_{i})^{2}=\frac{1}{n}\sum_{i=1}^{n}%
\tilde{Y}_{i}^{2}.  \label{Qn}
\end{equation}%
Using product formula (\ref{product}), we get
\begin{align}
Q_{n}-\mathbf{E}[Q_{n}]& =\frac{1}{n}\sum_{i=1}^{n}\left(
I_{2}^{W}(f_{i})^{2}-2\Vert f_{i}\Vert _{L^{2}([0,1]^{2})}^{2}\right)  \notag
\\
& =\frac{1}{n}\sum_{i=1}^{n}I_{4}^{W}(f_{i}\otimes f_{i})+\frac{4}{n}%
\sum_{i=1}^{n}I_{2}^{W}(f_{i}\otimes _{1}f_{i})  \notag \\
& =I_{4}^{W}\left( \frac{1}{n}\sum_{i=1}^{n}f_{i}\otimes f_{i}\right)
+I_{2}^{W}\left( \frac{4}{n}\sum_{i=1}^{n}f_{i}\otimes _{1}f_{i}\right)
\notag \\
& =:T_{4,n}+T_{2,n}.  \label{DecompQn}
\end{align}%
In the next section, we show that the asymptotic variance of $\sqrt{n}\left(
Q_{n}-\mathbf{E}[Q_{n}]\right) $ exits and we will compute its speed of
convergence. Then we establish a CLT for $Q_{n}$, and compute its Berry-Essé%
en speed of convergence in total variation.

\section{Asymptotic variance of the quadratic variation\label{Asymp}}

Using the orthogonality of multiple integrals living in different chaos, to
calculate the limiting variance of $\sqrt{n}(Q_{n}-\mathbf{E}[Q_{n}])$, we
need only study separately the second moments of the terms $T_{2,n}$ and $%
T_{4,n}$ given in (\ref{DecompQn}).

\subsection{Scale constant for $T_{2,n}$}

\begin{proposition}
\label{l1} Under Assumption (\ref{assumption}), with $T_{2,n}$ as in (\ref%
{DecompQn}), for large $n$,
\begin{equation}
\left\vert \mathbf{E}\left[ \left( \sqrt{n}T_{2,n}\right) ^{2}\right] -\frac{%
32\sum\limits_{\delta =1}^{\infty }\sigma _{\delta }^{4}}{(1-a_{1}^{2})^{2}}%
\right\vert \leqslant \frac{C_{1}}{n},  \label{C_{1}}
\end{equation}%
where $C_{1}:=32\left( \sum\limits_{\delta =1}^{\infty }\sigma _{\delta
}^{4}\right) \frac{\left[ 1+a_{1}^{2}(5+6a_{1}^{2})\right] }{%
(1-a_{1}^{4})^{2}(1-a_{1}^{2})}$. In particular
\begin{equation}
l_{1}:=\lim_{n\rightarrow \infty }\mathbf{E}\left[ \left( \sqrt{n}%
T_{2,n}\right) ^{2}\right] =\frac{32\sum\limits_{\delta =1}^{\infty }\sigma
_{\delta }^{4}}{(1-a_{1}^{2})^{2}}.  \label{l_{1}}
\end{equation}
\end{proposition}

\begin{proof}
We have $T_{2,n} = I^{W}_{2}(\frac{4}{n}\sum\limits_{i=1}^{n} f_{i}
\otimes_{1} f_{i})$, by the isometry property (\ref{Isometryprop}) of
multiple integrals, we get
\begin{equation}  \label{T_{2,n}moment}
\mathbf{E}[T_{2,n}^{2}] = 2 \| \frac{4}{n}\sum_{i=1}^{n} f_{i} \otimes_{1}
f_{i} \|^{2}_{L^{2}([0,1]^{2})} = \frac{32}{n^{2}} \sum_{i,j=1}^{n}
\left\langle f_{i} \otimes_{1} f_{i}, f_{j} \otimes_{1} f_{j}
\right\rangle_{L^{2}([0,1]^{2})}.
\end{equation}
Moreover, under Assumption (\ref{assumption}), we have
\begin{align*}
(f_{i} \otimes_{1} f_{i})(x,y) & = \sum\limits_{k_{1},k_{2}=1}^{i}
a_{1}^{i-k_{1}}
a_{1}^{i-k_{2}}\sum\limits_{\delta_{1},\delta_{2}=1}^{\infty}
\sigma_{\delta_{1}} \sigma_{\delta_{2}} (h_{k_{1},\delta_{1}}^{\otimes 2}
\otimes_{1}h_{k_{2},\delta_{2}}^{\otimes 2})(x,y) \\
& = \sum\limits_{k_{1},k_{2}=1}^{i} a_{1}^{i-k_{1}}
a_{1}^{i-k_{2}}\sum\limits_{\delta_{1},\delta_{2}=1}^{\infty}
\sigma_{\delta_{1}} \sigma_{\delta_{2}} (h_{{k_{1},\delta_{1}}} \otimes h_{{%
k_{2},\delta_{2}}})(x,y) \delta_{\delta_{1},\delta_{2}}\delta_{k_{1},k_{2}}
\\
& = \sum_{k=1}^{i} a_{1}^{2(i-k)} \sum_{\delta={1}}^{\infty}
\sigma_{\delta}^{2} (h_{k,\delta}^{\otimes 2 })(x,y).
\end{align*}
Therefore, for $i, j \geq 1$ such that $j \geq i$, we get
\begin{align*}
\left\langle f_{i} \otimes_{1} f_{i}, f_{j} \otimes_{1} f_{j}
\right\rangle_{L^{2}([0,1]^{2})} & = \sum_{k_{1}=1}^{i} a_{1}^{2(i-k_{1})}
\sum_{k_{2}=1}^{j} a_{1}^{2(j-k_{2})}
\sum\limits_{\delta_{1},\delta_{2}=1}^{\infty} \sigma_{\delta_{1}}^{2}
\sigma_{\delta_{2}}^{2} \left\langle h_{k_{1},\delta_{1}}^{\otimes 2},
h_{k_{2},\delta_{2}}^{\otimes 2} \right\rangle_{L^{2}([0,1]^{2})} \\
& = \sum_{k_{1}=1}^{i} a_{1}^{2(i-k_{1})} \sum_{k_{2}=1}^{j}
a_{1}^{2(j-k_{2})} \sum\limits_{\delta_{1},\delta_{2}=1}^{\infty}
\sigma_{\delta_{1}}^{2} \sigma_{\delta_{2}}^{2} \left(\left\langle
h_{k_{1},\delta_{1}}, h_{k_{2},\delta_{2}} \right\rangle_{L^{2}([0,1])}
\right)^{2} \\
& = \sum_{k_{1}=1}^{i} a_{1}^{2(i-k_{1})} \sum_{k_{2}=1}^{j}
a_{1}^{2(j-k_{2})} \sum\limits_{\delta_{1},\delta_{2}=1}^{\infty}
\sigma_{\delta_{1}}^{2} \sigma_{\delta_{2}}^{2}
\delta_{\delta_{1},\delta_{2}}\delta_{k_{1},k_{2}} \\
& = a_{1}^{2(j-i)} \sum_{k=1}^{i}a_{1}^{4(i-k)} \times
\sum\limits_{\delta=1}^{\infty} \sigma_{\delta}^{4} \\
& = \sum\limits_{\delta=1}^{\infty} \sigma_{\delta}^{4} \times
a_{1}^{2(j-i)} \left( \frac{1-a_{1}^{4i}}{1-a_{1}^{4}}\right).
\end{align*}
Therefore, by (\ref{T_{2,n}moment}), we have
\begin{align*}
\mathbf{E}\left[(\sqrt{n}T_{2,n})^{2}\right] & = \frac{32}{n}%
\sum\limits_{i=1}^{n} \|f_{i} \otimes_{1}f_{i}\|^{2}_{L^{2}([0,1]^{2})} +
\frac{64}{n} \sum\limits_{i=1}^{n-1}\sum\limits_{j=i+1}^{n} \left\langle
f_{i} \otimes_{1} f_{i}, f_{j} \otimes_{1} f_{j}
\right\rangle_{L^{2}([0,1]^{2})}.
\end{align*}
Moreover,
\begin{equation*}
\left|\frac{32}{n} \sum\limits_{i=1}^{n} \|f_{i}
\otimes_{1}f_{i}\|^{2}_{L^{2}([0,1]^{2})} - \frac{32 \sum\limits_{\delta =
1}^{\infty} \sigma_{\delta}^{4}}{(1-a_{1}^{4})} \right| \\
= \frac{32 \sum\limits_{\delta = 1}^{\infty} \sigma_{\delta}^{4}}{%
(1-a_{1}^{4})} \left|-\frac{1}{n} \sum\limits_{i =1}^{n}a_{1}^{4i} \right| \\
\leqslant \frac{32 \sum\limits_{\delta = 1}^{\infty} \sigma_{\delta}^{4}}{%
(1-a_{1}^{4})^2} \frac{1}{n}.
\end{equation*}
On the other hand
\begin{align*}
&\left|\frac{64}{n} \sum\limits_{i=1}^{n-1}\sum\limits_{j=i+1}^{n}
\left\langle f_{i} \otimes_{1} f_{i}, f_{j} \otimes_{1} f_{j}
\right\rangle_{L^{2}([0,1]^{2})} - \frac{64 a_{1}^{2} \sum\limits_{\delta =
1}^{\infty} \sigma_{\delta}^{4}}{(1-a_{1}^{4})(1-a_{1}^{2})} \right| \\
& = \left| \frac{64 \sum\limits_{\delta = 1}^{\infty} \sigma_{\delta}^{4}}{%
(1-a_{1}^{4})} \frac{1}{n} \sum\limits_{i=1}^{n-1}(1-a_{1}^{4i})
\sum\limits_{j=i+1}^{n} a_{1}^{2(j-i)} - \frac{64 a_{1}^{2}
\sum\limits_{\delta = 1}^{\infty} \sigma_{\delta}^{4} }{%
(1-a_{1}^{4})(1-a_{1}^{2})} \right| \\
& \leqslant \frac{64 a_{1}^{2} \sum\limits_{\delta = 1}^{\infty}
\sigma_{\delta}^{4} }{(1-a_{1}^{4})(1-a_{1}^{2})} \frac{1}{n}
\sum\limits_{i=1}^{n} \left| \left((1-a_{1}^{4i}) (1-a_{1}^{2(n-i)}) - 1
\right) \right| \\
& \leqslant \frac{64 a_{1}^{2} \sum\limits_{\delta = 1}^{\infty}
\sigma_{\delta}^{4} }{(1-a_{1}^{4})(1-a_{1}^{2})} \left( \frac{3}{n}
\sum\limits_{i=0}^{n} a_{1}^{2i} \right) \\
& \leqslant \frac{64 a_{1}^{2} \sum\limits_{\delta = 1}^{\infty}
\sigma_{\delta}^{4}}{(1-a_{1}^{4})(1-a_{1}^{2})^2} \frac{3}{n}.
\end{align*}
Consequently
\begin{align*}
\left| \mathbf{E}\left[\left(\sqrt{n}T_{2,n} \right)^{2}\right] - \frac{32
\sum\limits_{\delta=1}^{\infty}\sigma_{\delta}^{4}}{(1-a_{1}^{2})^{2}}
\right| & \leqslant \left|\frac{32}{n} \sum\limits_{i=1}^{n} \|f_{i}
\otimes_{1}f_{i}\|^{2}_{L^{2}([0,1]^{2})} - \frac{32 \sum\limits_{\delta =
1}^{\infty} \sigma_{\delta}^{4}}{(1-a_{1}^{4})} \right| \\
&\quad + \left|\frac{64}{n} \sum\limits_{i=1}^{n-1}\sum\limits_{j=i+1}^{n}
\left\langle f_{i} \otimes_{1} f_{i}, f_{j} \otimes_{1} f_{j}
\right\rangle_{L^{2}([0,1]^{2})} - \frac{64 a_{1}^{2} \sum\limits_{\delta =
1}^{\infty} \sigma_{\delta}^{4}}{(1-a_{1}^{4})(1-a_{1}^{2})} \right| \\
& \leqslant \frac{32 \sum\limits_{\delta = 1}^{\infty} \sigma_{\delta}^{4}}{%
(1-a_{1}^{4})^2} \frac{1}{n} + \frac{64 a_{1}^{2} \sum\limits_{\delta =
1}^{\infty} \sigma_{\delta}^{4}}{(1-a_{1}^{4})(1-a_{1}^{2})^2} \frac{3}{n} \\
& \leqslant 32 \sum\limits_{\delta = 1}^{\infty} \sigma_{\delta}^{4} \frac{%
\left[1 + a_{1}^2(5+6a_{1}^2)\right] }{(1-a_{1}^{4})^2(1-a_{1}^2)} \frac{1}{n%
}.
\end{align*}
The desired is therefore obtained.
\end{proof}

\subsection{Scale constant for $T_{4,n}$}

\begin{proposition}
\label{l2}Under Assumption (\ref{assumption}), with $T_{4,n}$ as in (\ref%
{DecompQn}), for large $n$,%
\begin{equation*}
\left\vert \mathbf{E}\left[ (\sqrt{n}T_{4,n})^{2}\right] -\frac{4}{%
(1-a_{1}^{2})^{2}}\left[ \sum\limits_{\delta =1}^{\infty }\sigma _{\delta
}^{4}+\left( \frac{1+a_{1}^{2}}{1-a_{1}^{2}}\right) \left( \sum_{\delta
=1}^{\infty }\sigma _{\delta }^{2}\right) ^{2}\right] \right\vert \leqslant
\frac{C_{2}}{n},
\end{equation*}%
where
\begin{equation}
C_{2}:=C_{2,1}+C_{2,2},\text{ \ }C_{2,1}:=4{\ \left( \sum\limits_{\delta
=1}^{\infty }\sigma _{\delta }^{2}\right) ^{2}}\frac{(3+17a_{1}^{2})}{%
(1-a_{1}^{2})^{4}}{\ \ }\text{ , }{\ \ }C_{2,2}:=4\left( {%
\sum\limits_{\delta =1}^{\infty }\sigma _{\delta }^{4}}\right) \frac{\left[
1+a_{1}^{2}(5+6a_{1}^{2})\right] }{(1-a_{1}^{4})^{2}(1-a_{1}^{2})}.
\label{C_{2}}
\end{equation}%
In particular
\begin{equation}
l_{2}:=\lim\limits_{n\rightarrow \infty }\mathbf{E}\left[ (\sqrt{n}%
T_{4,n})^{2}\right] =\frac{4}{(1-a_{1}^{2})^{2}}\left[ \sum\limits_{\delta
=1}^{\infty }\sigma _{\delta }^{4}+\left( \frac{1+a_{1}^{2}}{1-a_{1}^{2}}%
\right) \left( \sum_{\delta =1}^{\infty }\sigma _{\delta }^{2}\right) ^{2}%
\right] .  \label{l_{2}}
\end{equation}
\end{proposition}

\begin{proof}
By definition of the term $T_{4,n}$, we have
\begin{align*}
\mathbf{E}[(\sqrt{n}T_{4,n})^{2}]& =4!\Vert \frac{1}{\sqrt{n}}%
\sum_{i=1}^{n}f_{i}\tilde{\otimes}f_{i}\Vert _{L^{2}([0,1]^{4})}^{2} \\
& =\frac{4!}{n}\sum_{i,j=1}^{n}\left\langle f_{i}\tilde{\otimes}f_{i},f_{j}%
\tilde{\otimes}f_{j}\right\rangle _{L^{2}([0,1]^{4})},
\end{align*}%
where $f_{i}\tilde{\otimes}f_{i}$ denotes the symmetrization of $%
f_{i}\otimes f_{i}$, because the kernel $\sum_{i=1}^{n}f_{i}{\otimes }%
f_{i}\in L^{2}([0,1]^{4})$ is no longer symmetric. We deal with
symmetrization by using a combinatorial formula, obtaining
\begin{equation*}
4!\left\langle f_{i}\tilde{\otimes}f_{i},f_{j}\tilde{\otimes}%
f_{j}\right\rangle _{L^{2}([0,1]^{4})}=(2!)^{2}\left\langle f_{i}\otimes
f_{i},f_{j}\otimes f_{j}\right\rangle
_{L^{2}([0,1]^{4})}+(2!)^{2}\left\langle f_{i}\otimes _{1}f_{j},f_{j}\otimes
_{1}f_{i}\right\rangle _{L^{2}([0,1]^{2})}.
\end{equation*}%
Therefore
\begin{align*}
\mathbf{E}[(\sqrt{n}T_{4,n})^{2}]& =\frac{4}{n}\sum_{i,j=1}^{n}\left\langle
f_{i}\otimes f_{i},f_{j}\otimes f_{j}\right\rangle _{L^{2}([0,1]^{4})}+\frac{%
4}{n}\sum_{i,j=1}^{n}\left\langle f_{i}\otimes _{1}f_{j},f_{j}\otimes
_{1}f_{i}\right\rangle _{L^{2}([0,1]^{2})} \\
& =:T_{4,1,n}+T_{4,2,n}.
\end{align*}%
Moreover,
\begin{align}
T_{4,1,n}& =\frac{4}{n}\sum_{i,j=1}^{n}\left\langle f_{i}\otimes
f_{i},f_{j}\otimes f_{j}\right\rangle _{L^{2}([0,1]^{4})}  \notag \\
& =\frac{4}{n}\sum_{i,j=1}^{n}\left( \left\langle f_{i},f_{j}\right\rangle
_{L^{2}([0,1]^{2})}\right) ^{2}  \notag \\
& =\frac{4}{n}\sum_{i=1}^{n}\left( \left\langle f_{i},f_{i}\right\rangle
_{L^{2}([0,1]^{2})}\right) ^{2}+\frac{8}{n}\sum_{i=1}^{n-1}\sum_{j=i+1}^{n}%
\left( \left\langle f_{i},f_{j}\right\rangle _{L^{2}([0,1]^{2})}\right) ^{2}.
\label{T4.1.n}
\end{align}%
On the other hand, using (\ref{assumption}), we have for $j\geq i$
\begin{align*}
\left\langle f_{i},f_{j}\right\rangle _{L^{2}([0,1]^{2})}&
=\sum_{k_{1}=1}^{i}a_{1}^{i-k_{1}}\sum_{k_{2}=1}^{j}a_{1}^{j-k}\sum_{\delta
_{1},\delta _{2}=1}^{\infty }\sigma _{\delta _{1}}\sigma _{\delta
_{2}}\delta _{\delta _{1},\delta _{2}}\delta _{k_{1},k_{2}} \\
& =a_{1}^{(j-i)}\times \left( \sum_{\delta =1}^{\infty }\sigma _{\delta
}^{2}\right) \times \left( \frac{1-a_{1}^{2i}}{1-a_{1}^{2}}\right) .
\end{align*}%
Therefore by (\ref{T4.1.n}), we get
\begin{eqnarray*}
\left\vert T_{4,1,n}-\frac{4(1+a_{1}^{2})}{(1-a_{1}^{2})^{3}}\left(
\sum_{\delta =1}^{\infty }\sigma _{\delta }^{2}\right) ^{2}\right\vert
&\leqslant &\left\vert \frac{4}{n}\sum_{i=1}^{n}\Vert f_{i}\Vert
_{L^{2}([0,1]^{2})}^{4}-\frac{4}{(1-a_{1}^{2})^{2}}\left( \sum_{\delta
=1}^{\infty }\sigma _{\delta }^{2}\right) ^{2}\right\vert \\
&&\qquad +\left\vert \frac{8}{n}\sum_{i=1}^{n-1}\sum_{j=i+1}^{n}\left(
\left\langle f_{i},f_{j}\right\rangle _{L^{2}([0,1]^{2})}\right) ^{2}-\frac{%
8a_{1}^{2}}{(1-a_{1}^{2})^{3}}\left( \sum_{\delta =1}^{\infty }\sigma
_{\delta }^{2}\right) ^{2}\right\vert \\
&\leqslant &\frac{4}{(1-a_{1}^{2})^{2}}\left( \sum_{\delta =1}^{\infty
}\sigma _{\delta }^{2}\right) ^{2}\left\vert \frac{1}{n}%
\sum_{i=1}^{n}((1-a_{1}^{2i})^{2}-1)\right\vert \\
&&\qquad +\frac{8a_{1}^{2}}{(1-a_{1}^{2})^{3}}\left( \sum_{\delta
=1}^{\infty }\sigma _{\delta }^{2}\right) ^{2}\left\vert \frac{1}{n}%
\sum_{i=1}^{n}\left( (1-a_{1}^{2i})^{2}(1-a_{1}^{2(n-i)})-1\right)
\right\vert \\
&\leqslant &\frac{12}{(1-a_{1}^{2})^{3}}\left( \sum_{\delta =1}^{\infty
}\sigma _{\delta }^{2}\right) ^{2}\frac{1}{n}+\frac{80a_{1}^{2}}{%
(1-a_{1}^{2})^{4}}\left( \sum_{\delta =1}^{\infty }\sigma _{\delta
}^{2}\right) ^{2}\frac{1}{n} \\
&=&\frac{4\left( \sum\limits_{\delta =1}^{\infty }\sigma _{\delta
}^{2}\right) ^{2}}{(1-a_{1}^{2})^{4}}\frac{(3+17a_{1}^{2})}{n}.
\end{eqnarray*}%
Now let us estimate
\begin{equation*}
T_{4,2,n}=\frac{4}{n}\sum_{i,j=1}^{n}\left\langle f_{i}\otimes
_{1}f_{j},f_{j}\otimes _{1}f_{i}\right\rangle _{L^{2}([0,1]^{2})}.
\end{equation*}%
For $j\geq i$, $x,y\in L^{2}([0,1])$, we have
\begin{equation}
(f_{i}\otimes
_{1}f_{j})(x,y)=a_{1}^{(j-i)}\sum\limits_{k=1}^{i}a_{1}^{2(i-k)}\sum%
\limits_{\delta =1}^{\infty }\sigma _{\delta }^{2}(h_{k,\delta }^{\otimes
2})(x,y).  \label{ps}
\end{equation}%
So, for $j\geq i$,
\begin{equation*}
\left\langle f_{i}\otimes _{1}f_{j},f_{j}\otimes _{1}f_{i}\right\rangle
_{L^{2}([0,1]^{2})}=\int_{[0,1]^{2}}(f_{i}\otimes
_{1}f_{j})^{2}(x,y)dxdy=a_{1}^{2(j-i)}\left( \sum_{\delta =1}^{\infty
}\sigma _{\delta }^{4}\right) \times \left( \frac{1-a_{1}^{4i}}{1-a_{1}^{4}}%
\right) .
\end{equation*}%
Therefore,
\begin{align*}
\left\vert T_{4,2,n}-\frac{4\sum\limits_{\delta =1}^{\infty }\sigma _{\delta
}^{4}}{(1-a_{1}^{2})^{2}}\right\vert & =\left\vert \frac{4\sum\limits_{%
\delta =1}^{\infty }\sigma _{\delta }^{4}}{(1-a_{1}^{4})}\frac{1}{n}%
\sum_{i=1}^{n}(1-a_{1}^{4i})+\frac{8}{n}\frac{\sum\limits_{\delta
=1}^{\infty }\sigma _{\delta }^{4}}{(1-a_{1}^{4})}%
\sum_{i=1}^{n-1}(1-a_{1}^{4i})\sum_{j=i+1}^{n}a_{1}^{2(j-i)}-\frac{%
4\sum\limits_{\delta =1}^{\infty }\sigma _{\delta }^{4}}{(1-a_{1}^{2})^{2}}%
\right\vert \\
& \leqslant \left\vert -\frac{4\sum\limits_{\delta =1}^{\infty }\sigma
_{\delta }^{4}}{(1-a_{1}^{4})}\frac{1}{n}\sum_{i=1}^{n}a_{1}^{4i}\right\vert
+\left\vert \frac{8\sum\limits_{\delta =1}^{\infty }\sigma _{\delta
}^{4}a_{1}^{2}}{(1-a_{1}^{4})(1-a_{1}^{2})}\frac{1}{n}\sum_{i=1}^{n}\left(
(1-a_{1}^{4i})(1-a_{1}^{2(n-i)})-1\right) \right\vert \\
& \leqslant \frac{4\sum\limits_{\delta =1}^{\infty }\sigma _{\delta }^{4}}{%
(1-a_{1}^{4})^{2}}\frac{1}{n}+\frac{24a_{1}^{2}\sum\limits_{\delta
=1}^{\infty }\sigma _{\delta }^{4}}{(1-a_{1}^{4})(1-a_{1}^{2})^{2}}\frac{1}{n%
} \\
& =\frac{4\sum\limits_{\delta =1}^{\infty }\sigma _{\delta }^{4}}{%
(1-a_{1}^{4})^{2}(1-a_{1}^{2})}\frac{\left[ 1+a_{1}^{2}(5+6a_{1}^{2})\right]
}{n},
\end{align*}%
which completes the proof.
\end{proof}

To get a sense of how the two terms $T_{2,n}$ and $T_{4,n}$ compare to each
other, we propose the following example, which shows that, despite one's
best efforts, one should not expect either of these two terms to dominate
the other.

\begin{remark}
In the AR(1) model (\ref{ARmodel}) with chi-squared white noise, i.e. when $%
\sigma _{1}=\sigma $ and $\sigma _{\delta }=0$ for all $\delta \geq 2$, one
can try to compare the two formulas for the asymptotic variances of $T_{2,n}$
and $T_{4,n}$. Avoiding the situation where $\left\vert a_{1}\right\vert $
is very close to $1$, assuming for instance $|a_{1}|<2^{-1/2}$, so that $%
1-a_{1}^{2}>1/2$, when $n$ is large, we have
\begin{equation*}
Var(T_{2,n})\sim \frac{1}{n}\frac{32\sigma ^{4}}{(1-a_{1}^{2})^{2}}\sim
4\times (1-a_{1}^{2})\times Var(T_{4,n})>4\times \frac{1}{2}\times
Var(T_{4,n})=2\times Var(T_{4,n}).
\end{equation*}%
Therefore the sequence $T_{4,n}$ can be made to have a variance which is
significantly smaller that the one of $T_{2,n}$ in this case, but both of
them converge to zero at the same speed $n^{-1}$.
\end{remark}

Using the orthogonality between $T_{2,n}$ and $T_{4,n}$, Proposition \ref{l1}
and Proposition \ref{l2}, we conclude the following.

\begin{theorem}
\label{asymptoticvariance} Under Assumption (\ref{assumption}), with $Q_{n}$
as in (\ref{Qn}), for large $n$,
\begin{equation*}
\left\vert n~Var(Q_{n})-(l_{1}+l_{2})\right\vert \leqslant \frac{\left(
C_{1}+C_{2}\right) }{n},
\end{equation*}%
and in particular the asymptotic variance of $Q_{n}$ is
\begin{eqnarray*}
&&\lim\limits_{n\rightarrow \infty }n~Var(Q_{n})=l_{1}+l_{2} \\
&=&\frac{36}{(1-a_{1}^{2})^{2}}\left( \sum\limits_{\delta =1}^{\infty
}\sigma _{\delta }^{4}\right) +\frac{4(1+a_{1}^{2})}{(1-a_{1}^{2})^{3}}%
\left( \sum\limits_{\delta =1}^{\infty }\sigma _{\delta }^{2}\right) ^{2}
\end{eqnarray*}%
where $C_{1}$, $l_{1}$, and $C_{2}$, $l_{2}$ are given respectively in (\ref%
{C_{1}}), (\ref{l_{1}}), (\ref{C_{2}}) and (\ref{l_{2}}).
\end{theorem}

\begin{remark}
~

\begin{itemize}
\item From the previous theorem, we notice that for $n$ large, and fixed
values of the noise scale parameter family $\left\{ \sigma _{\delta },\delta
\geq 1\right\} $, the variance of $\sqrt{n}Q_{n}$ has high values when $%
|a_{1}|$ is close to 1, and approaches
\begin{equation*}
36\left( \sum\limits_{\delta =1}^{\infty }\sigma _{\delta }^{4}\right)
+4\left( \sum\limits_{\delta =1}^{\infty }\sigma _{\delta }^{2}\right) ^{2}
\end{equation*}%
when $\left\vert a_{1}\right\vert $ is small.

\item The previous theorem also shows that one can obtain other asymptotics
depending on the relation between $a_{1}$ and the family $\left\{ \sigma
_{\delta },\delta \geq 1\right\} $. For instance, when $|a_{1}|$ is close to
1, which is the limit of fast mean reversion, one can avoid an explosion of $%
Q_{n}$'s asymptotic variance by scaling the variance parameters
appropriately, leading to a fast-mean reversion and small noise regime.
Letting $1/\alpha :=1-a_{1}^{2}$, where $\alpha $ is interpreted as a rate
of mean reversion, one would only need to ensure that $\sum \sigma _{\delta
}^{4}=O\left( \alpha ^{-2}\right) $ and $\sum \sigma _{\delta }^{2}=O\left(
\alpha ^{-3/2}\right) $. In the example where there is a single non-zero
value $\sigma $, for instance, we would obtain for large $\alpha $,
\begin{equation*}
n\times Var(Q_{n})\sim 36\alpha ^{2}\sigma ^{4}+8\alpha ^{3}\sigma ^{4};
\end{equation*}%
here the second term dominates, and as $\alpha \rightarrow \infty $,
assuming $\alpha ^{3}\sigma ^{4}$ remains bounded, we would get an
asymptotic variance of $8\lim_{\alpha \rightarrow \infty }\alpha ^{3}\sigma
^{4}$ if the limit exists.
\end{itemize}
\end{remark}

\section{Berry-Esséen bound for the asymptotic normality of the
quadratic-variation\label{BerryEsseen}}

In this section, we prove that the quadratic variation defined in (\ref{Qn})
is asymptotically normal and we estimate the speed of this convergence in
total variation distance, showing it is of the Berry-Esséen-type order $%
n^{-1/2}$. For this aim, we will need the following theorem, which estimates
the total variation distance to the normal of the standardized sum of
variables in the 2nd and 4th chaos.

\begin{theorem}
\label{newbound} Let $F=I_{2}(f)+I_{4}(g)$ where $f\in L_{s}^{2}([0,1]^{2})$
and $g\in L_{s}^{2}([0,1]^{4})$. Then
\begin{eqnarray}
d_{TV}\left( \frac{F}{\sqrt{EF^{2}}},\mathcal{N}(0,1)\right) &\leqslant &%
\frac{4}{EF^{2}}\left[ \sqrt{2}\left\Vert f\otimes _{1}f\right\Vert
_{L^{2}([0,1]^{2})}+2\sqrt{6!}\left\Vert g\otimes _{1}g\right\Vert
_{L^{2}([0,1]^{6})}+18\sqrt{4!}\left\Vert g\otimes _{2}g\right\Vert
_{L^{2}([0,1]^{4})}\right.  \notag \\
&&\left. +36\sqrt{2}\left\Vert g\otimes _{3}g\right\Vert _{L^{2}([0,1]^{2})}
\right. +9\sqrt{2}\sqrt{\left\langle f\otimes f,g\otimes _{2}g\right\rangle
_{L^{2}([0,1]^{4})}}  \notag \\
&&\left. +3\sqrt{4!}\sqrt{\left\Vert f\otimes _{1}f\right\Vert
_{L^{2}([0,1]^{2})}\left\Vert g\otimes _{3}g\right\Vert _{L^{2}([0,1]^{2})}}%
\right] .  \label{Inegalite1}
\end{eqnarray}%
Moreover, letting $R_{F}$ be the bracketed term on the right-hand side of (%
\ref{Inegalite1}), for any constant $\sigma >0$, we have
\begin{equation}
d_{TV}\left( \frac{F}{\sigma },\mathcal{N}(0,1)\right) \leqslant \frac{4}{%
\sigma ^{2}}R_{F}+2\left\vert 1-\frac{EF^{2}}{\sigma ^{2}}\right\vert .
\label{Inegalite2}
\end{equation}
\end{theorem}

\begin{proof}
We have $F=I_{2}(f)+I_{4}(g)$. Then
\begin{equation*}
-D_{t}L^{-1}F=I_{1}(f(.,t))+I_{3}(g(.,t))=:u(t).
\end{equation*}%
Thus, using $F=\delta (-DL^{-1})F$, we can write $F=\delta (u)$. Now we use
the result of a simple calculation, labeled as (9) in the preprint \cite{NPY}
(see also \cite{NZ}), to obtain
\begin{eqnarray}
d_{TV}\left( \frac{F}{\sqrt{EF^{2}}},\mathcal{N}(0,1)\right) &\leqslant &%
\frac{2}{EF^{2}}\sqrt{Var\left( \langle DF,u\rangle _{L^{2}([0,1])}\right) }
\notag \\
&=&\frac{2}{EF^{2}}\sqrt{E\left( \langle DF,u\rangle
_{L^{2}([0,1])}-EF^{2}\right) ^{2}},  \label{newboundineq}
\end{eqnarray}%
where the last equality comes from the duality relation $E\langle
DF,u\rangle _{L^{2}([0,1])}=E(F\delta (u))=EF^{2}$. The prior inequality
appears to be used in a more general context here than what is stated in
\cite[Eq. (9)]{NPY}, but an immediate inspection of its proof therein shows
that it applies to any situation where $F=\delta (u)$, using only general
results such as Stein's lemma, the chain rule for the Malliavin derivative $%
D $, and the duality between $D$ and $\delta $.

On the other hand, using the product formula (\ref{product}),
\begin{eqnarray*}
\langle DF,u\rangle _{L^{2}([0,1])} &=&\int_{0}^{1}u(t)D_{t}Fdt \\
&=&\int_{0}^{1}u(t)(2I_{1}(f(.,t))+4I_{3}(g(.,t)))dt \\
&=&\int_{0}^{1}2\left[ I_{1}(f(.,t))\right] ^{2}+4\left[ I_{3}(g(.,t))\right]
^{2}+6I_{1}(f(.,t))I_{3}(g(.,t))\,dt \\
&=&\int_{0}^{1}\left[ 2I_{2}(f(.,t)\otimes f(.,t))+2\Vert f(.,t)\Vert
_{L^{2}([0,1])}^{2}+4I_{6}(g(.,t)\otimes g(.,t))\right. \\
&&\left. +36I_{4}(g(.,t)\otimes _{1}g(.,t))+72I_{2}(g(.,t)\otimes
_{2}g(.,t))+24g(.,t)\otimes _{3}g(.,t)\right. \\
&&\left. +6I_{4}(f(.,t)\otimes g(.,t))+18I_{2}(f(.,t)\otimes _{1}g(.,t))
\right] \,dt.
\end{eqnarray*}%
Thus
\begin{equation*}
\langle DF,u\rangle
_{L^{2}([0,1])}-EF^{2}=%
\int_{0}^{1}I_{2}(h_{1}(.,t))+I_{4}(h_{2}(.,t))+I_{6}(h_{3}(.,t))\,dt,
\end{equation*}%
where
\begin{eqnarray*}
h_{1}(.t) &=&2f(.,t)\otimes f(.,t)+72g(.,t)\otimes
_{2}g(.,t)+18f(.,t)\otimes _{1}g(.,t) \\
h_{2}(.t) &=&36g(.,t)\otimes _{1}g(.,t)+6f(.,t)\otimes g(.,t) \\
h_{3}(.t) &=&4g(.,t)\otimes g(.,t).
\end{eqnarray*}%
Therefore, using Minkowski inequality,
\begin{eqnarray*}
\sqrt{E\left( \langle DF,u\rangle _{L^{2}([0,1])}-EF^{2}\right) ^{2}}
&\leqslant &\sqrt{2}\left\Vert \int_{0}^{1}h_{1}(.,t)\,dt\right\Vert
_{L^{2}([0,1]^{2})}+\sqrt{4!}\left\Vert
\int_{0}^{1}h_{2}(.,t)\,dt\right\Vert _{L^{2}([0,1]^{4})} \\
&&+\sqrt{6!}\left\Vert \int_{0}^{1}h_{3}(.,t)\,dt\right\Vert
_{L^{2}([0,1]^{6})}.
\end{eqnarray*}%
Furthermore,
\begin{eqnarray*}
\left\Vert \int_{0}^{1}h_{1}(.,t)\,dt\right\Vert _{L^{2}([0,1]^{2})}
&\leqslant &\left\Vert 2\int_{0}^{1}f(.,t)\otimes f(.,t)\,dt\right\Vert
_{L^{2}([0,1]^{2})}+\left\Vert 72\int_{0}^{1}g(.,t)\otimes
_{2}g(.,t)\,dt\right\Vert _{L^{2}([0,1]^{2})} \\
&&+\left\Vert 18\int_{0}^{1}f(.,t)\otimes _{1}g(.,t)\,dt\right\Vert
_{L^{2}([0,1]^{2})} \\
&=&2\left\Vert f\otimes _{1}f\right\Vert _{L^{2}([0,1]^{2})}+72\left\Vert
g\otimes _{3}g\right\Vert _{L^{2}([0,1]^{2})}+18\sqrt{\left\langle f\otimes
f,g\otimes _{2}g\right\rangle _{L^{2}([0,1]^{4})}}.
\end{eqnarray*}%
Also,
\begin{eqnarray*}
\left\Vert \int_{0}^{1}h_{2}(.,t)\,dt\right\Vert _{L^{2}([0,1]^{4})}
&\leqslant &\left\Vert 36\int_{0}^{1}g(.,t)\otimes _{1}g(.,t)\,dt\right\Vert
_{L^{2}([0,1]^{2})}+\left\Vert 6\int_{0}^{1}f(.,t)\otimes
g(.,t)\,dt\right\Vert _{L^{2}([0,1]^{4})} \\
&=&36\left\Vert g\otimes _{2}g\right\Vert _{L^{2}([0,1]^{4})}+6\sqrt{%
\left\langle f\otimes _{1}f,g\otimes _{3}g\right\rangle _{L^{2}([0,1]^{2})}}
\\
&\leqslant &36\left\Vert g\otimes _{2}g\right\Vert _{L^{2}([0,1]^{4})}+6%
\sqrt{\left\Vert f\otimes _{1}f\right\Vert _{L^{2}([0,1]^{2})}\left\Vert
g\otimes _{3}g\right\Vert _{L^{2}([0,1]^{2})}},
\end{eqnarray*}%
and
\begin{eqnarray*}
\left\Vert \int_{0}^{1}h_{3}(.,t)\,dt\right\Vert _{L^{2}([0,1]^{6})}
&=&\left\Vert 4\int_{0}^{1}g(.,t)\otimes g(.,t)\,dt\right\Vert
_{L^{2}([0,1]^{6})} \\
&=&4\left\Vert g\otimes _{1}g\right\Vert _{L^{2}([0,1]^{6})}.
\end{eqnarray*}%
As a consequence,
\begin{eqnarray*}
\sqrt{E\left( \langle DF,u\rangle _{L^{2}([0,1])}-EF^{2}\right) ^{2}}
&\leqslant &2\sqrt{2}\left\Vert f\otimes _{1}f\right\Vert
_{L^{2}([0,1]^{2})}+72\sqrt{2}\left\Vert g\otimes _{3}g\right\Vert
_{L^{2}([0,1]^{2})} \\
&&+18\sqrt{2}\sqrt{\left\langle f\otimes f,g\otimes _{2}g\right\rangle
_{L^{2}([0,1]^{4})}}+36\sqrt{4!}\left\Vert g\otimes _{2}g\right\Vert
_{L^{2}([0,1]^{4})} \\
&&+6\sqrt{4!}\sqrt{\left\Vert f\otimes _{1}f\right\Vert
_{L^{2}([0,1]^{2})}\left\Vert g\otimes _{3}g\right\Vert _{L^{2}([0,1]^{2})}}%
+4\sqrt{6!}\left\Vert g\otimes _{1}g\right\Vert _{L^{2}([0,1]^{6})}.
\end{eqnarray*}%
This, combined with (\ref{newboundineq}), establishes inequality (\ref%
{Inegalite1}).

For (\ref{Inegalite2}), we have by (\ref{NPobs})
\begin{align}
d_{TV}\left( \frac{F}{\sigma },\mathcal{N}(0,1)\right) & \leqslant \frac{2}{%
\sigma ^{2}}\mathbf{E}[|\sigma ^{2}-\left\langle DF,u\right\rangle
_{L^{2}([0,1])}|]  \notag \\
& \leqslant \frac{2}{\sigma ^{2}}\mathbf{E}[|\mathbf{E}[F^{2}]-\left\langle
DF,u\right\rangle _{L^{2}([0,1])}|]+2\left\vert 1-\frac{\mathbf{E}[F^{2}]}{%
\sigma ^{2}}\right\vert .  \label{Inegalite3}
\end{align}%
Inequality (\ref{Inegalite2}) follows using inequalities (\ref{Inegalite3})
and (\ref{Inegalite1}).
\end{proof}

We will now use Theorem \ref{newbound} to prove that the quadratic variation
$Q_{n}$ satisfies the following Berry-Esséen theorem.

\begin{remark}
\label{KeyRem}It turns out that, when applying Theorem \ref{newbound} to
estimate the speed of convergence in the CLT for $Q_{n}$, the term $\sqrt{%
\left\langle f\otimes f,g\otimes _{2}g\right\rangle _{L^{2}([0,1]^{4})}}$
cannot merely be bounded above via Schwarz's inequality. See Lemma \ref%
{LemmaC5} below and its proof. This is the key element which allows us to
obtain the Berry-Esséen speed $n^{-1/2}$ in the next theorem.
\end{remark}

\begin{theorem}
\label{convergencenormalQn} With $Q_{n}$ defined in (\ref{Qn}), under
Assumption (\ref{assumption}), we have for all $n\geq 1$
\begin{equation}
d_{TV}\left( \frac{\sqrt{n}(Q_{n}-\mathbf{E}[Q_{n}])}{\sqrt{l_{1}+l_{2}}},%
\mathcal{N}(0,1)\right) \leqslant \frac{C_{0}}{\sqrt{n}},  \label{bornsupQn}
\end{equation}%
where
\begin{equation*}
C_{0}:=\frac{4\sqrt{2}}{l_{1}+l_{2}}\left( \frac{1}{2\sqrt{2}}%
(C_{1}+C_{2})+C_{3}+36C_{4,3}+9C_{5}+36\sqrt{3}C_{4,2}+6\sqrt{3}%
C_{3}^{1/2}C_{4,3}^{1/2}+12\sqrt{10}C_{4,1}\right) ,
\end{equation*}%
where $l_{1}$, $l_{2}$, are defined in the previous section in (\ref{l_{1}}%
), (\ref{l_{2}}), and $C_{1}$, $C_{2}$, $C_{3}$, $C_{4,r}$, $r=1,2,3$ and $%
C_{5}$ are given in the lemmas below, respectively in (\ref{C_{1}}), (\ref%
{C_{2}}), (\ref{C_{3}}), (\ref{C_{4}}) and (\ref{C_{5}}).

In particular $\sqrt{n}(Q_{n}-\mathbf{E}[Q_{n}])$ is asymptotically
Gaussian, namely
\begin{equation*}
\lim\limits_{n\rightarrow \infty }\sqrt{n}\left( Q_{n}-\mathbf{E}%
[Q_{n}]\right) \overset{law}{\longrightarrow }\mathcal{N}(0,l_{1}+l_{2}).
\end{equation*}
\end{theorem}

\begin{proof}
Based on the decomposition of $\left( Q_{n}-\mathbf{E}(Q_{n})\right) $ given
in (\ref{DecompQn}), we have
\begin{equation*}
\sqrt{n}\left( Q_{n}-\mathbf{E}[Q_{n}]\right) =\sqrt{n}T_{2,n}+\sqrt{n}%
T_{4,n}=I_{2}^{W}(g_{2,n})+I_{4}^{W}(g_{4,n}),
\end{equation*}%
where
\begin{equation}
g_{2,n}:=\frac{{4}}{\sqrt{n}}\sum_{i=1}^{n}f_{i}\otimes _{1}f_{i}\text{ \ }%
\text{and}\text{ \ }g_{4,n}:=\frac{1}{\sqrt{n}}\sum_{i=1}^{n}f_{i}\otimes
f_{i}.  \label{kernels}
\end{equation}%
Applying Theorem \ref{newbound} to $\sqrt{n}\left( Q_{n}-\mathbf{E}%
[Q_{n}]\right) $, we get
\begin{align}
& d_{TV}\left( \frac{\sqrt{n}(Q_{n}-\mathbf{E}[Q_{n}])}{\sqrt{l_{1}+l_{2}}},%
\mathcal{N}(0,1)\right)  \notag \\
& \leqslant \frac{4}{l_{1}+l_{2}}\left[ \sqrt{2}\left\Vert g_{2,n}\otimes
_{1}g_{2,n}\right\Vert _{L^{2}([0,1]^{2})}+36\sqrt{2}\left\Vert
g_{4,n}\otimes _{3}g_{4,n}\right\Vert _{L^{2}([0,1]^{2})}\right.  \notag \\
& \quad +9\sqrt{2}\sqrt{\left\langle g_{2,n}\otimes g_{2,n},g_{4,n}\otimes
_{2}g_{4,n}\right\rangle _{L^{2}([0,1]^{4})}}  \notag \\
& \quad \left. +18\sqrt{4!}\left\Vert g_{4,n}\otimes _{2}g_{4,n}\right\Vert
_{L^{2}([0,1]^{4})}+3\sqrt{4!}\sqrt{\left\Vert g_{2,n}\otimes
_{1}g_{2,n}\right\Vert _{L^{2}([0,1]^{2})}\left\Vert g_{4,n}\otimes
_{3}g_{4,n}\right\Vert _{L^{2}([0,1]^{2})}}\right.  \notag \\
& \quad \left. +2\sqrt{6!}\left\Vert g_{4,n}\otimes _{1}g_{4,n}\right\Vert
_{L^{2}([0,1]^{6})}\right] +2\left\vert 1-\frac{\mathbf{E}\left[ \left(
\sqrt{n}\left( Q_{n}-\mathbf{E}[Q_{n}]\right) \right) ^{2}\right] }{%
(l_{1}+l_{2})}\right\vert .  \label{decom.interm}
\end{align}%
We study first the contractions of the kernels $g_{2,n}$ and $g_{4,n}$ given
in (\ref{kernels}) and we prove that they satisfy the following lemmas.

\begin{lemma}
\label{cong2} If Assumption (\ref{assumption}) holds, the kernel $g_{2,n}$
defined in (\ref{kernels}) satisfies%
\begin{equation}
\Vert g_{2,n}\otimes _{1}g_{2,n}\Vert _{L^{2}([0,1]^{2})}\leqslant \frac{%
C_{3}}{\sqrt{n}},  \label{C_{3}}
\end{equation}%
where $C_{3}:=\sqrt{4!\frac{4^{4}}{n}\left( \sum_{\delta =1}^{\infty }\sigma
_{\delta }^{8}\right) \frac{1}{1-a_{1}^{8}}\left( \frac{1}{1-a_{1}^{2}}%
\right) ^{3}}$.
\end{lemma}

\begin{proof}
We have
\begin{align*}
\| g_{2,n} \otimes_{1} g_{2,n}\|^{2}_{L^{2}([0,1]^{2})}& = \frac{4^{4}}{n^{2}%
} \|\sum_{i,j=1}^{n} (f_{i} \otimes_{1}f_{i}) \otimes_{1} (f_{j}
\otimes_{1}f_{j}) \|^{2}_{L^{2}([0,1]^{2})} \\
&= \frac{4^{4}}{n^{2}} \sum_{i,j,k,l=1}^{n} \left \langle (f_{i}
\otimes_{1}f_{i}) \otimes_{1} (f_{j} \otimes_{1}f_{j}), (f_{k}
\otimes_{1}f_{k}) \otimes_{1} (f_{l} \otimes_{1}f_{l}) \right
\rangle_{L^{2}([0,1]^{2})}.
\end{align*}
Moreover, by above calculations and (\ref{assumption})
\begin{equation*}
(f_{i} \otimes_{1}f_{i})(x,y) = \sum\limits_{k=1}^{i} a_{1}^{2(i-k)}
\sum\limits_{\delta = 1}^{\infty} \sigma_{\delta}^{2}
h_{k,\delta}^{\otimes_{2}}(x,y).
\end{equation*}
Hence
\begin{align*}
& (f_{i} \otimes_{1}f_{i}) \otimes_{1} (f_{j} \otimes_{1}f_{j})(x,y) \\
& = \int_{[0,1]} (f_{i} \otimes_{1}f_{i})(x,t) (f_{j} \otimes_{1}f_{j})(y,t)
dt \\
& = \sum\limits_{k_{1}=1}^{i} a_{1}^{2(i-k_{1})} \sum\limits_{k_{2}=1}^{j}
a_{1}^{2(j-k_{2})} \sum\limits_{\delta_1,\delta_2 =1 }^{\infty}
\sigma_{\delta_1}^{2} \sigma_{\delta_2}^{2} \int_{[0,1]}
h_{k_{1},\delta_{1}}(x)h_{k_{1},\delta_{1}}(t) h_{k_{2},\delta_{2}}(y)
h_{k_{2},\delta_{2}}(t) dt \\
& = \sum\limits_{k_{1}=1}^{i} a_{1}^{2(i-k_{1})} \sum\limits_{k_{2}=1}^{j}
a_{1}^{2(j-k_{2})} \sum\limits_{\delta =1}^{\infty} \sigma_{\delta}^{4}
(h_{k_{1},\delta} \otimes h_{k_{2},\delta})(x,y) \delta_{k_{1},k_{2}} \\
& = \sum\limits_{k_{1}=1}^{i \wedge j} a_{1}^{2(i+j-2k_{1})} \left(
\sum\limits_{\delta =1}^{+ \infty} \sigma_{\delta}^{4} \right)
(h_{k_{1},\delta} \otimes h_{k_{1},\delta})(x,y).
\end{align*}
Similarly
\begin{equation*}
(f_{k} \otimes_{1}f_{k}) \otimes_{1} (f_{l} \otimes_{1}f_{l})(x,y) =
\sum\limits_{k_{2}=1}^{k \wedge l} a_{1}^{2(k+l-2k_{2})} \left(
\sum\limits_{\delta =1}^{+ \infty} \sigma_{\delta}^{4} \right)
(h_{k_{2},\delta} \otimes h_{k_{2},\delta})(x,y).
\end{equation*}
Therefore
\begin{align*}
& \left \langle (f_{i} \otimes_{1}f_{i}) \otimes_{1} (f_{j}
\otimes_{1}f_{j}), (f_{k} \otimes_{1}f_{k}) \otimes_{1} (f_{l}
\otimes_{1}f_{l}) \right \rangle_{L^{2}([0,1]^{2})} \\
& = \int_{[0,1]^{2}} \left((f_{i} \otimes_{1}f_{i}) \otimes_{1} (f_{j}
\otimes_{1}f_{j})\right)(x,y)\left((f_{k} \otimes_{1}f_{k}) \otimes_{1}
(f_{l} \otimes_{1}f_{l})\right)(x,y) dx dy \\
& = \sum\limits_{m=1}^{i \wedge j \wedge k \wedge l} a_{1}^{2(i+j+k+l-4m)}
\left( \sum\limits_{\delta =1}^{\infty} \sigma_{\delta}^{8}\right).
\end{align*}
Consequently
\begin{eqnarray*}
\| g_{2,n} \otimes_{1} g_{2,n}\|^{2}_{L^{2}([0,1]^{2})} & = &\frac{4^4}{n^2}%
\sum_{i,j,k,l=1}^n\sum_{r=1}^{i\wedge j\wedge k\wedge l}
a_{1}^{2(i+j+k+l-4r)} \sum_{\delta = 1}^{\infty} \sigma_{\delta}^8 \\
&\leqslant&4!\frac{4^4}{n^2}\left(\sum_{\delta = 1}^{\infty}
\sigma_{\delta}^8\right)\sum_{1\leqslant i\leqslant j\leqslant k\leqslant
l\leqslant n}\sum_{r=1}^{i} a_{1}^{2(i+j+k+l-4r)} \\
&\leqslant&4!\frac{4^4}{n^2}\left(\sum_{\delta = 1}^{\infty}
\sigma_{\delta}^8\right)\sum_{1\leqslant i\leqslant j\leqslant k\leqslant
l\leqslant n} a_{1}^{2(j+k+l-3i)} \frac{1-a_1^{8i}}{1-a_1^{8}} \\
&\leqslant&4!\frac{4^4}{n^2}\left(\sum_{\delta = 1}^{\infty}
\sigma_{\delta}^8\right)\frac{1}{1-a_1^{8}}\sum_{1\leqslant i\leqslant
j\leqslant k\leqslant l\leqslant n} a_{1}^{2(j+k+l-3i)} \\
&\leqslant&4!\frac{4^4}{n}\left(\sum_{\delta = 1}^{\infty}
\sigma_{\delta}^8\right)\frac{1}{1-a_1^{8}}\sum_{ k_1,k_2,k_3=0}^{n}
a_{1}^{2(k_1+k_2+k_3)} \\
&=&4!\frac{4}{n}\left(\sum_{\delta = 1}^{\infty} \sigma_{\delta}^8\right)%
\frac{1}{1-a_1^{8}}\left(\sum_{k_1=0}^{n} a_{1}^{2k_1} \right)^3 \\
&\leqslant&4!\frac{4^4}{n}\left(\sum_{\delta = 1}^{\infty}
\sigma_{\delta}^8\right)\frac{1}{1-a_1^{8}}\left(\frac{1}{1-a_1^{2}}%
\right)^3,
\end{eqnarray*}
where we used the change of variables $k_1=j-i,\, k_2=k-i,\, k_3=l-i$. The
desired result therefore follows.
\end{proof}

\begin{lemma}
If Assumption (\ref{assumption}) holds, for every $r=1,2,3,$ the kernel $%
g_{4,n}$ defined in (\ref{kernels}) satisfies%
\begin{equation}
\Vert g_{4,n}{\otimes _{r}}g_{4,n}\Vert _{L^{2}([0,1]^{(8-2r)})}\leqslant
\frac{{C}_{4,r}}{\sqrt{n}},  \label{C_{4}}
\end{equation}%
where
\begin{equation*}
{C}_{4,r}:=\left\{
\begin{array}{lll}
\sqrt{4!\left( \sum\limits_{\delta =1}^{\infty }\sigma _{\delta }^{2}\right)
^{2}\left( \sum\limits_{\delta =1}^{\infty }\sigma _{\delta }^{4}\right)
\left( \frac{1}{(1-a_{1}^{2})(1-a_{1}^{4})}\right) ^{3}} & \mbox{ if }\text{
\ }r=1, &  \\
~~ &  &  \\
\sqrt{{4!}{\left( \sum\limits_{\delta =1}^{\infty }\sigma _{\delta
}^{2}\right) ^{4}}\left( \frac{1}{1-a_{1}^{2}}\right) ^{7}} & \mbox{ if }%
\text{ \ }r=2, &  \\
~~ &  &  \\
\sqrt{{4!}\left( \sum\limits_{\delta =1}^{\infty }\sigma _{\delta
}^{2}\right) ^{2}\left( \sum\limits_{\delta =1}^{\infty }\sigma _{\delta
}^{4}\right) \frac{1}{(1-a_{1}^{4})^{2}}\frac{1}{(1-a_{1}^{2})^{4}}} & %
\mbox{ if }\text{ \ }r=3. &
\end{array}%
\right.
\end{equation*}
\end{lemma}

\begin{proof}
For $r=1,2,3$, we have
\begin{align*}
\Vert g_{4,n}{\otimes _{r}}g_{4,n}\Vert _{L^{2}([0,1]^{2(4-r)})}^{2}& =\frac{%
1}{n^{2}}\Vert \sum_{i,j=1}^{n}(f_{i}\otimes f_{i})\tilde{\otimes _{r}}%
(f_{j}\otimes f_{j})\Vert _{L^{2}([0,1]^{2(4-r)})}^{2} \\
& \leqslant \frac{1}{n^{2}}\Vert \sum_{i,j=1}^{n}(f_{i}\otimes f_{i}){%
\otimes _{r}}(f_{j}\otimes f_{j})\Vert _{L^{2}([0,1]^{2(4-r)})}^{2} \\
& =\frac{1}{n^{2}}\sum_{i,j,k,l=1}^{n}\left\langle (f_{i}\otimes f_{i}){%
\otimes _{r}}(f_{j}\otimes f_{j}),(f_{k}\otimes f_{k}){\otimes _{r}}%
(f_{l}\otimes f_{l})\right\rangle _{L^{2}([0,1]^{2(4-r)})}.
\end{align*}

For $r=1$, we get
\begin{align}
\Vert g_{4,n}{\otimes _{1}}g_{4,n}\Vert _{L^{2}([0,1]^{6})}^{2}& \leqslant
\frac{1}{n^{2}}\sum_{i,j,k,l=1}^{n}\left\langle (f_{i}\otimes f_{i}){\otimes
_{1}}(f_{j}\otimes f_{j}),(f_{k}\otimes f_{k}){\otimes _{1}}(f_{l}\otimes
f_{l})\right\rangle _{L^{2}([0,1]^{6})}  \notag \\
& =\frac{1}{n^{2}}\sum_{i,j,k,l=1}^{n}\left\langle f_{i},f_{k}\right\rangle
_{L^{2}([0,1]^{2})}\left\langle f_{j},f_{l}\right\rangle
_{L^{2}([0,1]^{2})}\left\langle f_{i}\otimes _{1}f_{j},f_{k}\otimes
_{1}f_{l}\right\rangle _{L^{2}([0,1]^{2})}  \notag \\
& =\frac{4!}{n^{2}}\sum_{1\leqslant i\leqslant j\leqslant k\leqslant
l\leqslant n}\left\langle f_{i},f_{k}\right\rangle
_{L^{2}([0,1]^{2})}\left\langle f_{j},f_{l}\right\rangle
_{L^{2}([0,1]^{2})}\left\langle f_{i}\otimes _{1}f_{j},f_{k}\otimes
_{1}f_{l}\right\rangle _{L^{2}([0,1]^{2})}.  \label{estimg4}
\end{align}%
By (\ref{assumption}), for all $1\leqslant i,k\leqslant n$,

\begin{equation}
\left\langle f_{i},f_{k}\right\rangle
_{L^{2}([0,1]^{2})}=\sum\limits_{m=1}^{i\wedge k}a_{1}^{i+k-2m}\times \left(
\sum\limits_{\delta =1}^{\infty }\sigma _{\delta }^{2}\right) .
\label{produitscalaire}
\end{equation}%
Similarly, for all $1\leqslant j,l\leqslant n$,
\begin{equation*}
\left\langle f_{j},f_{l}\right\rangle
_{L^{2}([0,1]^{2})}=\sum\limits_{m=1}^{i\wedge k}a_{1}^{j+l-2m}\times \left(
\sum\limits_{\delta =1}^{\infty }\sigma _{\delta }^{2}\right) .
\end{equation*}%
On the other hand, for all $1\leqslant i,j\leqslant n$
\begin{eqnarray}
(f_{i}\otimes _{1}f_{j})(x,y) &=&\int_{0}^{1}f_{i}(x,t)f_{j}(y,t)dt  \notag
\\
&=&\sum_{m=1}^{i\wedge j}a_{1}^{i+j-2m}\sum_{\delta =1}^{\infty }\sigma
_{\delta }^{2}h_{m,\delta }^{\otimes 2}(x,y).  \label{contr1}
\end{eqnarray}%
Hence, by (\ref{assumption}), for all $1\leqslant i,j,k,l\leqslant n$,
\begin{equation*}
\left\langle f_{i}\otimes _{1}f_{j},f_{k}\otimes _{1}f_{l}\right\rangle
_{L^{2}([0,1]^{2})}=\sum\limits_{m=1}^{i\wedge j\wedge k\wedge
l}a_{1}^{(i+j+l+k-4m)}\times \left( \sum\limits_{\delta =1}^{\infty }\sigma
_{\delta }^{4}\right) .
\end{equation*}%
Therefore, from (\ref{estimg4}) and above calculations, we have
\begin{align*}
& \Vert g_{4,n}{\otimes _{1}}g_{4,n}\Vert _{L^{2}([0,1]^{6})}^{2} \\
& \leqslant \frac{4!}{(1-a_{1}^{2})^{2}}\left( \sum\limits_{\delta
=1}^{\infty }\sigma _{\delta }^{2}\right) ^{2}\left( \sum\limits_{\delta
=1}^{\infty }\sigma _{\delta }^{4}\right) \frac{1}{(1-a_{1}^{4})}\frac{1}{%
n^{2}}\sum_{1\leqslant i\leqslant j\leqslant k\leqslant l\leqslant
n}a_{1}^{2(k-i)}a_{1}^{2(l-i)}(1-a_{1}^{2i})(1-a_{1}^{4i})(1-a_{1}^{2j}) \\
& \leqslant \frac{4!}{(1-a_{1}^{2})^{2}}\left( \sum\limits_{\delta
=1}^{\infty }\sigma _{\delta }^{2}\right) ^{2}\left( \sum\limits_{\delta
=1}^{\infty }\sigma _{\delta }^{4}\right) \frac{1}{(1-a_{1}^{4})}\frac{1}{n}%
\sum\limits_{k_{1},k_{2},k_{3}=0}^{n}a_{1}^{2(2k_{1}+2k_{2}+k_{3})} \\
& =\frac{4!}{(1-a_{1}^{2})^{2}}\frac{1}{(1-a_{1}^{4})}\left(
\sum\limits_{\delta =1}^{\infty }\sigma _{\delta }^{2}\right) ^{2}\left(
\sum\limits_{\delta =1}^{\infty }\sigma _{\delta }^{4}\right) \frac{1}{n}%
\left( \sum\limits_{r_{1}=0}^{n}a_{1}^{4r_{1}}\right) ^{2}\left(
\sum\limits_{r_{2}=0}^{n}a_{1}^{2r_{2}}\right) \\
& \leqslant {4!}\left( \sum\limits_{\delta =1}^{\infty }\sigma _{\delta
}^{2}\right) ^{2}\left( \sum\limits_{\delta =1}^{\infty }\sigma _{\delta
}^{4}\right) \left( \frac{1}{(1-a_{1}^{2})(1-a_{1}^{4})}\right) ^{3}\times
\frac{1}{n},
\end{align*}%
where we used the change of variables $k_{1}=j-i$, $k_{2}=k-j$, $k_{3}=l-k$.

For $r=2$, we have
\begin{align*}
\Vert g_{4,n}{\otimes _{2}}g_{4,n}\Vert _{L^{2}([0,1]^{4})}^{2}& \leqslant
\frac{1}{n^{2}}\sum_{i,j,k,l=1}^{n}\left\langle (f_{i}\otimes f_{i}){\otimes
_{2}}(f_{j}\otimes f_{j}),(f_{k}\otimes f_{k}){\otimes _{2}}(f_{l}\otimes
f_{l})\right\rangle _{L^{2}([0,1]^{4})} \\
& =\frac{1}{n^{2}}\sum_{i,j,k,l=1}^{n}\left\langle f_{i},f_{j}\right\rangle
_{L^{2}([0,1]^{2})}\left\langle f_{k},f_{l}\right\rangle
_{L^{2}([0,1]^{2})}\left\langle f_{i},f_{k}\right\rangle
_{L^{2}([0,1]^{2})}\left\langle f_{j},f_{l}\right\rangle _{L^{2}([0,1]^{2})}
\\
& =\frac{4!}{n^{2}}\sum_{1\leqslant i\leqslant j\leqslant k\leqslant
l\leqslant n}\left\langle f_{i},f_{j}\right\rangle
_{L^{2}([0,1]^{2})}\left\langle f_{k},f_{l}\right\rangle
_{L^{2}([0,1]^{2})}\left\langle f_{i},f_{k}\right\rangle
_{L^{2}([0,1]^{2})}\left\langle f_{j},f_{l}\right\rangle _{L^{2}([0,1]^{2})}.
\end{align*}%
Hence, by (\ref{produitscalaire}) and (\ref{assumption}), we get
\begin{align*}
\Vert g_{4,n}{\otimes _{2}}g_{4,n}\Vert _{L^{2}([0,1]^{4})}^{2}& \leqslant
\frac{4!}{(1-a_{1}^{2})^{4}}\frac{\left( \sum\limits_{\delta =1}^{\infty
}\sigma _{\delta }^{2}\right) ^{4}}{n^{2}}\sum_{1\leqslant i\leqslant
j\leqslant k\leqslant l\leqslant
n}a_{1}^{2(l-i)}(1-a_{1}^{2i})^{2}(1-a_{1}^{2k})(1-a_{1}^{2j}) \\
& \leqslant \frac{4!}{(1-a_{1}^{2})^{4}}\frac{\left( \sum\limits_{\delta
=1}^{\infty }\sigma _{\delta }^{2}\right) ^{4}}{n}\sum%
\limits_{k_{1},k_{2},k_{3}=0}^{n}a_{1}^{2(k_{1}+k_{2}+k_{3})} \\
& =\frac{4!}{(1-a_{1}^{2})^{4}}\frac{\left( \sum\limits_{\delta =1}^{\infty
}\sigma _{\delta }^{2}\right) ^{4}}{n}\left(
\sum\limits_{k_{1}=0}^{n}a_{1}^{2k_{1}}\right) ^{3} \\
& \leqslant \frac{4!}{(1-a_{1}^{2})^{7}}\frac{\left( \sum\limits_{\delta
=1}^{\infty }\sigma _{\delta }^{2}\right) ^{4}}{n},
\end{align*}%
where we used the change of variables $k_{1}=j-i$, $k_{2}=k-j$, $k_{3}=l-k$.

For $r=3$, we have
\begin{align*}
& \Vert g_{4,n}{\otimes _{3}}g_{4,n}\Vert _{L^{2}([0,1]^{2})}^{2} \\
& \leqslant \frac{1}{n^{2}}\sum_{i,j,k,l=1}^{n}\left\langle (f_{i}\otimes
f_{i}){\otimes _{3}}(f_{j}\otimes f_{j}),(f_{k}\otimes f_{k}){\otimes _{3}}%
(f_{l}\otimes f_{l})\right\rangle _{L^{2}([0,1]^{2})} \\
& =\frac{1}{n^{2}}\sum_{i,j,k,l=1}^{n}\left\langle f_{i},f_{j}\right\rangle
_{L^{2}([0,1]^{2})}\left\langle f_{k},f_{l}\right\rangle
_{L^{2}([0,1]^{2})}\left\langle f_{i}\otimes _{1}f_{j},f_{k}\otimes
_{1}f_{l}\right\rangle _{L^{2}([0,1]^{2})} \\
& =\frac{4!}{n^{2}}\sum_{1\leqslant i\leqslant j\leqslant k\leqslant
l\leqslant n}\left\langle f_{i},f_{j}\right\rangle
_{L^{2}([0,1]^{2})}\left\langle f_{k},f_{l}\right\rangle
_{L^{2}([0,1]^{2})}\left\langle f_{i}\otimes _{1}f_{j},f_{k}\otimes
_{1}f_{l}\right\rangle _{L^{2}([0,1]^{2})} \\
& \leqslant \frac{4!}{(1-a_{1}^{2})^{2}}\left( \sum\limits_{\delta
=1}^{\infty }\sigma _{\delta }^{2}\right) ^{2}\left( \sum\limits_{\delta
=1}^{\infty }\sigma _{\delta }^{4}\right) \frac{1}{(1-a_{1}^{4})}\frac{1}{%
n^{2}}\sum_{1\leqslant i\leqslant j\leqslant k\leqslant l\leqslant
n}a_{1}^{2(j-i)}a_{1}^{2(l-i)}(1-a_{1}^{2i})(1-a_{1}^{4i})(1-a_{1}^{2k}) \\
& \leqslant \frac{4!}{(1-a_{1}^{2})^{2}}\left( \sum\limits_{\delta
=1}^{\infty }\sigma _{\delta }^{2}\right) ^{2}\left( \sum\limits_{\delta
=1}^{\infty }\sigma _{\delta }^{4}\right) \frac{1}{(1-a_{1}^{4})}\frac{1}{n}%
\sum\limits_{k_{1},k_{2},k_{3}=0}^{n}a_{1}^{2(2k_{1}+k_{2}+k_{3})} \\
& =\frac{4!}{(1-a_{1}^{2})^{2}}\frac{1}{(1-a_{1}^{4})}\left(
\sum\limits_{\delta =1}^{\infty }\sigma _{\delta }^{2}\right) ^{2}\left(
\sum\limits_{\delta =1}^{\infty }\sigma _{\delta }^{4}\right) \frac{1}{n}%
\left( \sum\limits_{r_{1}=0}^{n}a_{1}^{4r_{1}}\right) \left(
\sum\limits_{r_{2}=0}^{n}a_{1}^{2r_{2}}\right) ^{2} \\
& \leqslant {4!}\left( \sum\limits_{\delta =1}^{\infty }\sigma _{\delta
}^{2}\right) ^{2}\left( \sum\limits_{\delta =1}^{\infty }\sigma _{\delta
}^{4}\right) \frac{1}{(1-a_{1}^{4})^{2}}\frac{1}{(1-a_{1}^{2})^{4}}\times
\frac{1}{n}
\end{align*}%
where, we used (\ref{produitscalaire}) and the change of variables $%
k_{1}=j-i $, $k_{2}=k-j$, $k_{3}=l-k$, which ends the proof.
\end{proof}

\begin{lemma}
\label{LemmaC5}Suppose Assumption (\ref{assumption}) holds. Consider the
kernels $g_{2,n}$ and $g_{4,n}$ defined in (\ref{kernels}), then we have
\begin{equation}
\sqrt{\left\langle g_{2,n}\otimes g_{2,n},g_{4,n}\otimes
_{2}g_{4,n}\right\rangle _{L^{2}([0,1]^{4})}}\leqslant \frac{C_{5}}{\sqrt{n}}%
,  \label{C_{5}}
\end{equation}%
where $C_{5}:=\sqrt{4!\frac{4}{n}\left( \sum_{\delta =1}^{\infty }\sigma
_{\delta }^{8}\right) \frac{1}{1-a_{1}^{8}}\left( \frac{1}{1-a_{1}^{2}}%
\right) ^{3}}$.
\end{lemma}

\begin{proof}
We have
\begin{eqnarray*}
&&\left\langle g_{2,n}\otimes g_{2,n},g_{4,n}\otimes_{2}g_{4,n}
\right\rangle_{L^2([0,1]^4)} \\
&=&\frac{4}{n^2}\sum_{i,j,k,l=1}^n\left\langle (f_{i}
\otimes_{1}f_{i})\otimes (f_{j} \otimes_{1}f_{j}), (f_{k} \otimes
f_{k})\otimes_{2} (f_{l} \otimes f_{l}) \right\rangle_{L^2([0,1]^4)} \\
&=&\frac{4}{n^2}\sum_{i,j,k,l=1}^n\int_{[0,1]^4}(f_{i}%
\otimes_{1}f_{k})(x_1,x_3)(f_{i}\otimes_{1}f_{k})(x_1,x_4)
(f_{j}\otimes_{1}f_{l})(x_2,x_3)(f_{j}%
\otimes_{1}f_{l})(x_2,x_4)dx_1dx_2dx_3dx_4.
\end{eqnarray*}
Consequently, using (\ref{contr1}), we get
\begin{eqnarray*}
\left\langle g_{2,n}\otimes g_{2,n},g_{4,n}\otimes_{2}g_{4,n}
\right\rangle_{L^2([0,1]^4)} &\leqslant&\frac{4}{n^2}\sum_{i,j,k,l=1}^n%
\sum_{r=1}^{i\wedge j\wedge k\wedge l} a_{1}^{2(i+j+k+l-4r)} \sum_{\delta =
1}^{\infty} \sigma_{\delta}^8 \\
&\leqslant&4!\frac{4}{n^2}\left(\sum_{\delta = 1}^{\infty}
\sigma_{\delta}^8\right)\sum_{1\leqslant i\leqslant j\leqslant k\leqslant
l\leqslant n}\sum_{r=1}^{i} a_{1}^{2(i+j+k+l-4r)} \\
&\leqslant&4!\frac{4}{n^2}\left(\sum_{\delta = 1}^{\infty}
\sigma_{\delta}^8\right)\sum_{1\leqslant i\leqslant j\leqslant k\leqslant
l\leqslant n} a_{1}^{2(j+k+l-3i)} \frac{1-a_1^{8i}}{1-a_1^{8}} \\
&\leqslant&4!\frac{4}{n^2}\left(\sum_{\delta = 1}^{\infty}
\sigma_{\delta}^8\right)\frac{1}{1-a_1^{8}}\sum_{1\leqslant i\leqslant
j\leqslant k\leqslant l\leqslant n} a_{1}^{2(j+k+l-3i)} \\
&\leqslant&4!\frac{4}{n}\left(\sum_{\delta = 1}^{\infty}
\sigma_{\delta}^8\right)\frac{1}{1-a_1^{8}}\sum_{ k_1,k_2,k_3=0}^{n}
a_{1}^{2(k_1+k_2+k_3)} \\
&=&4!\frac{4}{n}\left(\sum_{\delta = 1}^{\infty} \sigma_{\delta}^8\right)%
\frac{1}{1-a_1^{8}}\left(\sum_{k_1=0}^{n} a_{1}^{2k_1} \right)^3 \\
&\leqslant&4!\frac{4}{n}\left(\sum_{\delta = 1}^{\infty}
\sigma_{\delta}^8\right)\frac{1}{1-a_1^{8}}\left(\frac{1}{1-a_1^{2}}%
\right)^3,
\end{eqnarray*}
where we used the change of variables $k_1=j-i,\, k_2=k-i,\, k_3=l-i$.
\end{proof}

The bound (\ref{bornsupQn}) is then a direct consequence of inequality (\ref%
{decom.interm}) and the estimates given respectively in (\ref{C_{3}}), (\ref%
{C_{4}}), (\ref{C_{5}}) and Theorem \ref{asymptoticvariance}.
\end{proof}

\section{Application: estimation of the mean-reversion parameter\label{Estim}%
}

In this section, to illustrate the implications of Theorem \ref%
{convergencenormalQn} in parameter estimation in an easily tractable case,
we consider that we have observations $Y_{n}$ coming from a specific version
of our second-chaos AR(1) model (\ref{ARmodel}), that which is driven by a
chi-squared white noise with one degree\ of freedom:%
\begin{equation}
\left\{
\begin{array}{ll}
Y_{n}=a_{0}+a_{1}Y_{n-1}+\sigma (Z_{n}^{2}-1), & n\geq 1 \\
~~ &  \\
Y_{0}=y_{0}\in \mathbb{R}. &
\end{array}%
\right.  \label{ARmodelsimple}
\end{equation}%
where $a_{0}$, $a_{1}$ and $\sigma $ are real constants, and $\{Z_{n},n\geq
1\}$ are i.i.d. standard normal. This is model (\ref{ARmodel}) where all $%
\sigma _{\delta }$'s are zero except for the first one.

\begin{proposition}
\label{MoyenneQn}The quadratic variation $Q_{n}$ defined in (\ref{Qn}) for
model (\ref{ARmodelsimple}) satisfies, for all $n\geq 1$,%
\begin{equation*}
\left\vert \mathbf{E}[Q_{n}]-\frac{2\sigma ^{2}}{(1-a_{1}^{2})}\right\vert
\leqslant \frac{C_{6}}{n},
\end{equation*}%
where $C_{6}:=\frac{2\sigma ^{2}}{(1-a_{1}^{2})^{2}}$.
\end{proposition}

\begin{proof}
From the definition of $Q_{n}$ in (\ref{Qn}), we have by the isometry
property of multiple integrals
\begin{align*}
& \left\vert \mathbf{E}[Q_{n}]-\frac{2\sigma ^{2}}{(1-a_{1}^{2})}\right\vert
=\left\vert \frac{2}{n}\sum\limits_{i=1}^{n}\Vert f_{i}\Vert
_{L^{2}([0,1]^{2})}^{2}-\frac{2\sigma ^{2}}{1-a_{1}^{2}}\right\vert \\
& =\left\vert \frac{2\sigma ^{2}}{n}\sum\limits_{i=1}^{n}\sum%
\limits_{k=1}^{i}a_{1}^{2(i-k)}-\frac{2\sigma ^{2}}{1-a_{1}^{2}}\right\vert
=\left\vert \frac{2\sigma ^{2}}{1-a_{1}^{2}}\left( \frac{1}{n}%
\sum_{i=1}^{n}(1-a_{1}^{2i})-1\right) \right\vert \\
& =\left\vert -\frac{2\sigma ^{2}}{1-a_{1}^{2}}\frac{1}{n}%
\sum\limits_{i=1}^{n}a_{1}^{2i}\right\vert \leqslant \frac{2\sigma ^{2}}{%
(1-a_{1}^{2})^{2}}\frac{1}{n}.
\end{align*}
\end{proof}

\begin{remark}
Assuming that $\sigma $ is known, Proposition (\ref{MoyenneQn}) shows that
the quadratic variation $Q_{n}$ is an asymptotically unbiased estimator for $%
2\sigma ^{2}/(1-a_{1}^{2})$, and thus, after a transformation, for $%
\left\vert a_{1}\right\vert $ as well:%
\begin{equation*}
\lim_{n\rightarrow \infty }\sqrt{1-\frac{2\sigma ^{2}}{\mathbf{E}[Q_{n}]}}%
=\left\vert a_{1}\right\vert .
\end{equation*}
\end{remark}

Therefore, using the fact that $\mathbf{E}[Q_{n}]$ can be estimated via $%
Q_{n}$, we suggest the following moment estimator for the mean-reversion
rate $|a_{1}|$
\begin{equation}
\hat{a}_{n}:=f(Q_{n}),  \label{estima1}
\end{equation}%
where
\begin{equation}
f(x):=\sqrt{1-\frac{2\sigma ^{2}}{x}},\quad x>2\sigma ^{2}.  \label{f}
\end{equation}

\subsection{Properties of the estimator $\hat{a}_{n}$}

\begin{proposition}
The estimator $\hat{a}_{n}$ of the mean reversion parameter $|a_{1}|$
defined in (\ref{estima1}) is strongly consistent, namely almost surely
\begin{equation*}
\lim_{n\rightarrow \infty }\hat{a}_{n}=|a_{1}|.
\end{equation*}
\end{proposition}

\begin{proof}
We write $Q_{n}=\frac{V_{n}}{\sqrt{n}}+\mathbf{E}[Q_{n}]$, with $V_{n}=\sqrt{%
n}(Q_{n}-\mathbf{E}[Q_{n}])$. According to Proposition (\ref%
{asymptoticvariance}), we have $\mathbf{E}[V_{n}^{2}]\rightarrow l_{1}+l_{2}$%
, $n\rightarrow \infty $. Hence, there exists a constant $C>0$, such that $%
n\geq 1$
\begin{equation*}
\left( \mathbf{E}\left[ \left\vert \frac{V_{n}}{\sqrt{n}}\right\vert ^{2}%
\right] \right) ^{1/2}\leqslant \frac{C}{\sqrt{n}}.
\end{equation*}%
Hence, by Lemma \ref{Borel-Cantelli} we have almost surely $\frac{V_{n}}{%
\sqrt{n}}\rightarrow 0$, as $n\rightarrow \infty $. On the other hand, by
Proposition \ref{MoyenneQn}, $\mathbf{E}[Q_{n}]\rightarrow \frac{2\sigma ^{2}%
}{1-a_{1}^{2}}$, as $n\rightarrow \infty $. Thus $Q_{n}\rightarrow \frac{%
2\sigma ^{2}}{1-a_{1}^{2}}$ almost surely as $n\rightarrow \infty $, as
announced.
\end{proof}

\begin{proposition}
\label{GaussianEst} Under Assumption (\ref{assumption}), the estimator $\hat{%
a}_{n}$ defined in (\ref{estima1}) satsifies
\begin{equation*}
d_{W}\left( \frac{\sqrt{n}}{f^{\prime }(\mu )\sqrt{l_{1}+l_{2}}}\left( \hat{a%
}_{n}-|a_{1}|\right) ,\mathcal{N}(0,1)\right) \unlhd n^{-1/2}.
\end{equation*}%
where $l_{1}$ and $l_{2}$ are given in Propositions \ref{l1} and \ref{l2}
respectively and $\mu =\frac{2\sigma ^{2}}{(1-a_{1}^{2})}$.

In particular $\hat{a}_{n}$ is asymptotically Gaussian; more precisely we
have as $n\rightarrow +\infty $
\begin{equation*}
\sqrt{n}\left( \hat{a}_{n}-|a_{1}|\right) \longrightarrow \mathcal{N}\left(
0,f^{\prime }(\mu )^{2}\times (l_{1}+l_{2})\right) ,
\end{equation*}%
where
\begin{equation*}
f^{\prime }(\mu )^{2}\times (l_{1}+l_{2})=\frac{(1-a_{1}^{2})(5-4a_{1}^{2})}{%
2a_{1}^{2}}.
\end{equation*}
\end{proposition}

\begin{proof}
For $x>2\sigma ^{2}$, the function $f$ defined in (\ref{f}) has an inverse $%
f^{-1}(y)=\frac{2\sigma ^{2}}{1-y^{2}}$. Let denote $\mu =f^{-1}(|a_{1}|)=%
\frac{2\sigma ^{2}}{(1-a_{1}^{2})}>2\sigma ^{2}$, for all $a_{1}\in (-1,1)$.
On the other hand, we can write
\begin{equation*}
\frac{\sqrt{n}}{\sqrt{l_{1}+l_{2}}}\left( Q_{n}-\mu \right) =\frac{\sqrt{n}}{%
\sqrt{l_{1}+l_{2}}}\left( Q_{n}-\mathbf{E}[Q_{n}]\right) +\frac{\sqrt{n}}{%
\sqrt{l_{1}+l_{2}}}\left( \mathbf{E}[Q_{n}]-\mu \right)
\end{equation*}%
Therefore, from the properties of the Wasserstein metric and denoting $N\sim
\mathcal{N}(0,1)$, we get
\begin{align}
d_{W}\left( \frac{\sqrt{n}}{\sqrt{l_{1}+l_{2}}}\left( Q_{n}-\mu \right)
,N\right) & \leqslant \frac{\sqrt{n}}{\sqrt{l_{1}+l_{2}}}\left\vert \mathbf{E%
}[Q_{n}]-\mu \right\vert +d_{W}\left( \frac{\sqrt{n}}{\sqrt{l_{1}+l_{2}}}%
\left( Q_{n}-\mathbf{E}[Q_{n}])\right) ,N\right)  \notag \\
& \leqslant \frac{C_{6}}{\sqrt{l_{1}+l_{2}}}n^{-1/2}+C_{0}n^{-1/2}  \notag \\
& \unlhd n^{-1/2}.  \label{boundQn}
\end{align}%
where we used the bound (\ref{bornsupQn}) and Proposition \ref{MoyenneQn}
for the above bounds. On the other hand, since the function $f$ defined in (%
\ref{f}) is a diffeomorphism and since $\hat{a}_{n}=f(Q_{n})$, then by the
mean-value theorem, there exists a random variable $\xi _{n}$ $\in $ $%
[|Q_{n},\mu |]$ such that%
\begin{equation*}
\left( \hat{a}_{n}-|a_{1}|\right) =f^{\prime }(\xi _{n})\left( Q_{n}-\mu
\right) .
\end{equation*}%
We have%
\begin{align*}
d_{W}\left( \frac{\sqrt{n}}{f^{\prime }(\mu )\sqrt{l_{1}+l_{2}}}\left( \hat{a%
}_{n}-|a_{1}|\right) ,\mathcal{N}(0,1)\right) & \leqslant d_{W}\left( \frac{%
\sqrt{n}}{f^{\prime }(\mu )\sqrt{l_{1}+l_{2}}}\left( \hat{a}%
_{n}-|a_{1}|\right) ,\frac{\sqrt{n}}{\sqrt{l_{1}+l_{2}}}\left( Q_{n}-\mu
\right) \right) \\
& \quad \quad +d_{W}\left( \frac{\sqrt{n}}{\sqrt{l_{1}+l_{2}}}\left(
Q_{n}-\mu \right) ,\mathcal{N}(0,1)\right)
\end{align*}%
According to (\ref{boundQn}), the last term is bounded by the speed $%
n^{-1/2} $. Moreover, applying the mean-value theorem again since $f$ is
twice continuously differentiable, there exists a random variable $\delta
_{n}$ $\in $ $[|\xi _{n},\mu |]$ $\subset $ $[|Q_{n},\mu |]$, such that
\begin{eqnarray}
&&d_{W}\left( \frac{\sqrt{n}}{f^{\prime }(\mu )\sqrt{l_{1}+l_{2}}}\left(
\hat{a}_{n}-|a_{1}|\right) ,\frac{\sqrt{n}}{\sqrt{l_{1}+l_{2}}}\left(
Q_{n}-\mu \right) \right) |f^{\prime }(\mu )|\sqrt{l_{1}+l_{2}}  \notag \\
\qquad &\leqslant &\sqrt{n}\mathbf{E}[|(Q_{n}-\mu )(f^{\prime }(\xi
_{n})-f^{\prime }(\mu ))|]  \notag \\
&=&\sqrt{n}\mathbf{E}[\left\vert (Q_{n}-\mu )f^{\prime \prime }(\delta
_{n})(\xi _{n}-\mu )\right\vert ]  \notag \\
&\leqslant &\sqrt{n}\mathbf{E}[\left( Q_{n}-\mu )^{2}\right) |f^{\prime
\prime }(\delta _{n})|]  \notag \\
&\leqslant &\sqrt{n}\left( \mathbf{E}[(Q_{n}-\mu )^{2p}]\right) ^{1/p}\left(
\mathbf{E}[|f^{\prime \prime }(\delta _{n})|^{p^{\prime }}]\right)
^{1/p^{\prime }},  \label{borne0}
\end{eqnarray}%
where we used Hölder's inequality with $p$, $p^{\prime }$ are two reals
greater than 1 such that $\frac{1}{p}+\frac{1}{p^{\prime }}=1$. Moreover, by
the hypercontractivity property for multiple integrals (\ref%
{hypercontractivity}), there exists a constant $C(p)$ such that
\begin{align}
\sqrt{n}\left( \mathbf{E}[(Q_{n}-\mu )^{2p}]\right) ^{1/p}& \leqslant C(p)%
\sqrt{n}\mathbf{E}[(Q_{n}-\mu )^{2}]  \notag \\
& \leqslant 2C(p)\sqrt{n}\mathbf{E}[(Q_{n}-\mathbf{E}[Q_{n}])^{2}]+2C(p)%
\sqrt{n}(\mathbf{E}[Q_{n}]-\mu )^{2}  \notag \\
& \leqslant Cn^{-1/2}+Cn^{-3/2}  \notag \\
& \unlhd n^{-1/2},  \label{premiereborne}
\end{align}%
where we used the inequality $(a+b)^{2}\leqslant 2a^{2}+2b^{2}$, for all $a$%
, $b$ $\in \mathbb{R}$ and the bounds of Proposition \ref{MoyenneQn} and
Theorem \ref{asymptoticvariance} respectively. On the other hand, for all $%
x>2\sigma ^{2}$,
\begin{equation*}
f^{^{\prime \prime }}(x)=-\frac{2\sigma ^{2}}{x^{3}}\left( 1-\frac{3\sigma
^{2}}{2x}\right) \left( 1-\frac{2\sigma ^{2}}{x}\right) ^{-3/2}<0.
\end{equation*}%
Therefore using (\ref{borne0}) and (\ref{premiereborne}), to obtain a bound
for the term $d_{W}\left( \frac{\sqrt{n}}{f^{\prime }(\mu )\sqrt{l_{1}+l_{2}}%
}\left( \hat{a}_{n}-|a_{1}|\right) ,\frac{\sqrt{n}}{\sqrt{l_{1}+l_{2}}}%
\left( Q_{n}-\mu \right) \right) $, it remains to show that $\mathbf{E}%
[|f^{\prime \prime }(\delta _{n})|^{p^{\prime }}]$ is finite for some $%
p^{\prime }>1$ but using the fact that $\delta _{n}\in \lbrack |Q_{n},\mu |]$
and the monotonicity of $f^{\prime \prime }$, it is actually sufficient to
show that for some $p^{\prime }>1$, we have
\begin{equation*}
\sup_{n\geq 1}\mathbf{E}\left[ |f^{\prime \prime }(x)|^{p^{\prime }}\right]
=\sup_{n\geq 1}\mathbf{E}\left[ |2\sigma ^{2}Q_{n}-3\sigma ^{4}|^{p^{\prime
}}|Q_{n}|^{-5p^{\prime }/2}|Q_{n}-2\sigma ^{2}|^{-3p^{\prime }/2}\right]
<\infty .
\end{equation*}%
The function $|f^{\prime \prime }|$ has two singularities in $0$ and in $%
2\sigma ^{2}$ and thus is not bounded. But, we can write
\begin{equation*}
\mathbf{E}\left[ |f^{\prime \prime }(Q_{n})|\right] =\mathbf{E}\left[
|f^{\prime \prime }(Q_{n})|\mathbf{1}_{\{|Q_{n}-\mu |\geq \frac{1}{\sqrt{n}}%
\}}\right] +\mathbf{E}\left[ |f^{\prime \prime }(Q_{n})|\mathbf{1}%
_{\{|Q_{n}-\mu |<\frac{1}{\sqrt{n}}\}}\right] .
\end{equation*}%
For the term $\mathbf{E}\left[ |f^{\prime \prime }(Q_{n})|\mathbf{1}%
_{\{|Q_{n}-\mu |<\frac{1}{\sqrt{n}}\}}\right] $ and since $\mu >2\sigma ^{2}$%
, we can pick $n$ such that $\frac{1}{\sqrt{n}}<\sigma ^{2}$. Then $%
Q_{n}>2\sigma ^{2}-\sigma ^{2}>0$, therefore $Q_{n}$ is bounded away from $0$
and the term $|Q_{n}|^{-5p^{\prime }/2}$ has no singularity for any $%
p^{\prime }>1$. For the term $|Q_{n}-2\sigma ^{2}|^{-3p^{\prime }/2}$ , we
put $C:=\frac{\mu -2\sigma ^{2}}{2}$, the constant $C\neq 0$, because $\mu
\neq 2\sigma ^{2}\Leftrightarrow a_{1}\neq 0$, we can assume $a_{1}\neq 0$,
because there is no AR(1) process with $a_{1}=0$. Therefore, we can pick $n$
such that $\frac{1}{\sqrt{n}}<C$. In this case $Q_{n}-2\sigma ^{2}=Q_{n}-\mu
+2C>-\frac{1}{\sqrt{n}}+2C>2C-C>0$, hence the term $|Q_{n}-2\sigma
^{2}|^{-3p^{\prime }/2}$ has no singularities at $2\sigma ^{2}$ for any $%
p^{\prime }>1$. In conclusion, to avoid the singularities at both $0$ and $%
2\sigma ^{2}$, it is sufficient to pick $n$ such that

$\frac{1}{\sqrt{n}}<\sigma ^{2}\wedge C=%
\begin{cases}
\sigma ^{2} & \text{ if }|a_{1}|\geq \frac{1}{\sqrt{2}} \\
C & \text{ if }|a_{1}|\leqslant \frac{1}{\sqrt{2}},%
\end{cases}%
$.

For the other term, by the asymptotic normality of $\sqrt{n}(Q_{n}-\mu )$,
we get as $n\sim +\infty $ and for $p^{\prime }>1$
\begin{align*}
\mathbf{E}\left[ |f^{\prime \prime }(Q_{n})|\mathbf{1}_{\{|Q_{n}-\mu |\geq
\frac{1}{\sqrt{n}}\}}\right] & \sim |2\sigma ^{2}|^{p^{\prime
}}\int_{1}^{+\infty }|\frac{z}{\sqrt{n}}+\mu -\frac{3}{2}\sigma
^{2}|^{p^{\prime }}|\frac{z}{\sqrt{n}}+\mu |^{-5p^{\prime }/2}|\frac{z}{%
\sqrt{n}}+\mu -2\sigma ^{2}|^{-3p^{\prime }/2}e^{-z^{2}/2}dz \\
& \leqslant |2\sigma ^{2}|^{p^{\prime }}\int_{1}^{+\infty }|\frac{z}{\sqrt{n}%
}+\mu -2\sigma ^{2}|^{-7p^{\prime }/2}e^{-z^{2}/2}dz \\
& \leqslant \frac{\sqrt{2\pi }}{2}|2\sigma ^{2}|^{p^{\prime }}(\mu -2\sigma
^{2})^{-7p^{\prime }/2}.
\end{align*}%
which gives the desired result.
\end{proof}

\subsection{Numerical Results}

The table below reports the mean and standard deviation of the proposed
estimator $\hat{a}_{n}$ defined in (\ref{estima1}) of the true value of the
mean-reversion parameter $|a_{1}|$. {\renewcommand{\arraystretch}{1.3}
\begin{table}[H]
\centering
\vspace{0.3cm}
\begin{tabular}{lcccccccc}
\cline{1-9}
& \multicolumn{2}{c}{$|a_{1}|=0.10$} & \multicolumn{2}{c}{$|a_{1}|=0.30$} &
\multicolumn{2}{c}{$|a_{1}|=0.50$} & \multicolumn{2}{c}{$|a_{1}|=0.70$} \\
& Mean & Std dev & Mean & Std dev & Mean & Std dev & Mean & Std dev \\
\cline{1-9}
$n = 3000$ & 0.2178 & 0.0901 & 0.2887 & 0.0969 & 0.4946 & 0.0520 & 0.6962 &
0.0234 \\
$n = 5000$ & 0.1878 & 0.0866 & 0.2905 & 0.0616 & 0.4966 & 0.0413 & 0.6978 &
0.0270 \\
$n = 10000$ & 0.1630 & 0.0692 & 0.2928 & 0.0852 & 0.4974 & 0.0315 & 0.6987 &
0.0215 \\ \cline{1-9}
\end{tabular}%
\end{table}
}\label{table1} We simulate the values of the estimator $\hat{a}_{n}$ from
the quadratic variation $Q_{n}$ for different sample sizes $n$ and for fixed
$\sigma ^{2}$ chosen to be equal to 1. For each sample size $n$, the mean
and the standard deviation are obtained by 500 replications. The table above
confirms that the estimator $\hat{a}_{n}$ is strongly consistent even for
small values of $n$ and has small standard deviations for different true
values of $|a_{1}|$. Moreover, the estimator $\hat{a}_{n}$ is more efficient
for values of $|a_{1}|$ greater than $0.5$, this could be explained by the
fact that the asymptotic variance of limiting law of $|\hat{a}_{n}|$ is $%
\frac{(1-a_{1}^{2})(5-4a_{1}^{2})}{2a_{1}^{2}}$, which is high for small
values of $|a_{1}|$ and small for values of $|a_{1}|$ close to 1. Therefore,
$\hat{a}_{n}$ is presumably more accurate as an estimator when $|a_{1}|$ is
closer to 1, e.g. greater than 0.5 as can be seen in the figure below.
\begin{figure}[H]
\centering
\includegraphics[width=0.7\linewidth]{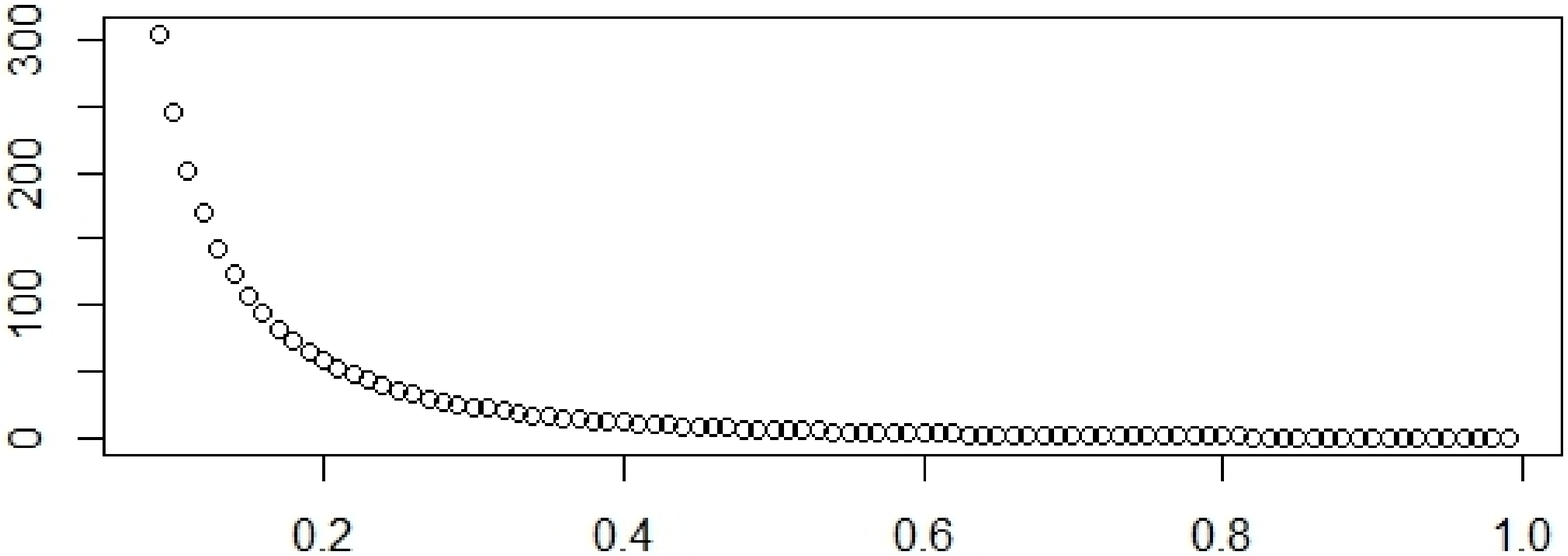}
\caption{Asymptotic variance for values of $|a_{1}|$ between 0.09 and 0.99}
\label{figurevariance}
\end{figure}

To investigate the asymptotic distribution of $\hat{a}_{n}$ empirically, we
need to compare the distribution of the following statistic
\begin{equation}
\phi (n,a_{1}):=\frac{\sqrt{2a_{1}^{2}}}{\sqrt{(1-a_{1}^{2})(5-4a_{1}^{2})}}%
\sqrt{n}(\hat{a}_{n}-|a_{1}|)  \label{stats}
\end{equation}%
with the standard normal distribution $\mathcal{N}(0,1)$. For this aim, for
parameter choices $\left\vert a_{1}\right\vert =0.5$, $n=3000$, $\sigma =1$,
and based on 3000 replications, we obtained the following histogram:
\begin{figure}[H]
\centering
\includegraphics[width=0.7\linewidth]{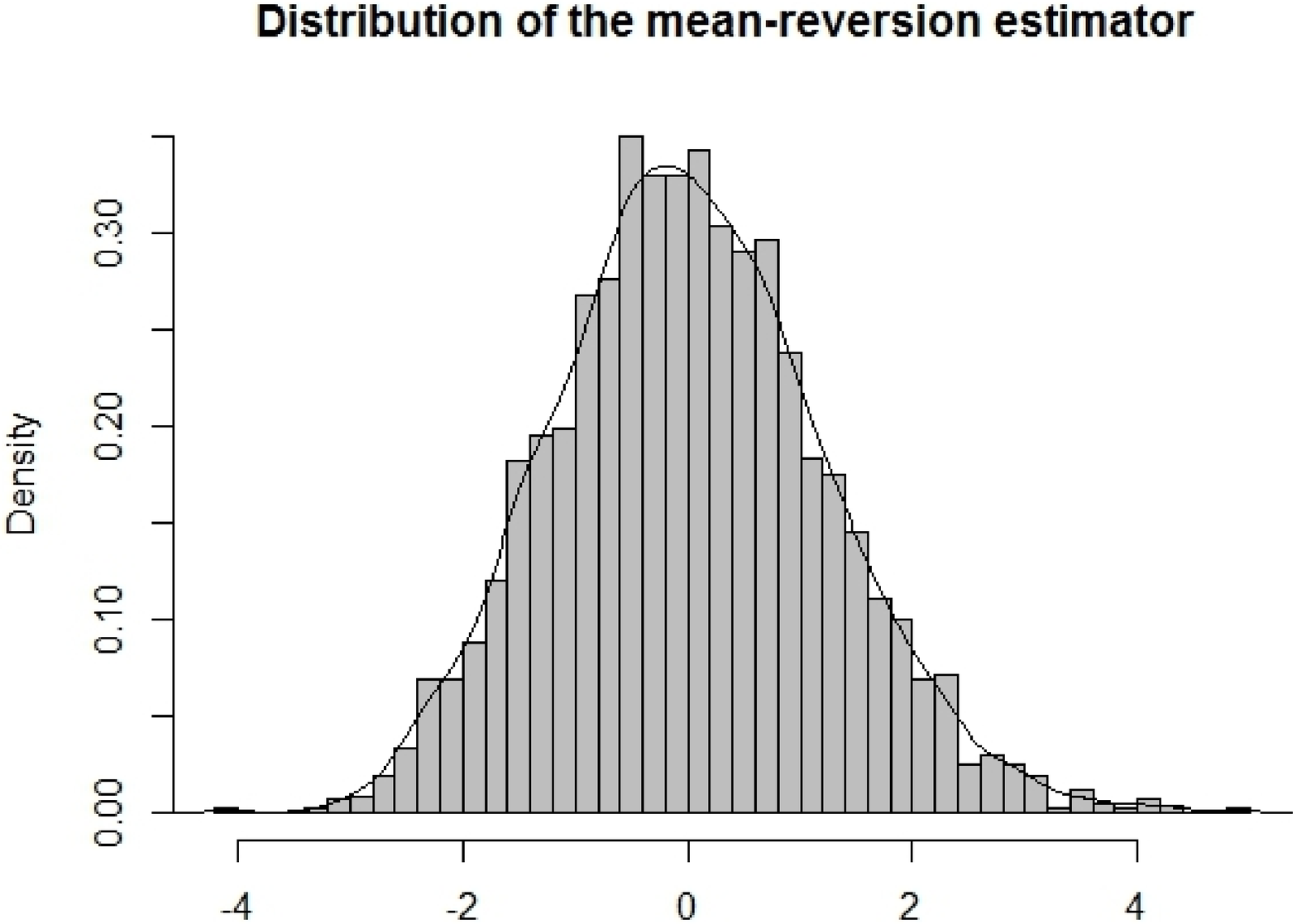}
\caption{Histogram of $\protect\phi (n,a_{1})$ with $n=3000$, $|a_{1}|=0.5,%
\protect\sigma =1$, 3000 replications.}
\label{fig:ar1}
\end{figure}
This Figure \ref{fig:ar1} shows that the normal approximation of the
distribution of the statistic $\phi (n,a_{1})$ is reasonable even if the
sampling size $n$ is not very large. The table below compares statistics of $%
\phi (n,a_{1})$ and $\mathcal{N}$(0,1) based on 3000 replications, with $%
n=3000$, and $\sigma =1$. The empirical mean, median and standard deviation
of $\phi (n,a_{1})$ match those of $\mathcal{N}$(0,1) very closely,
corroborating our theoretical results.

\begin{center}
\begin{tabular}{cccc}
\multicolumn{4}{c}{} \\
Statistics & Mean & Median & Standard Deviation \\ \hline
$\mathcal{N}$(0,1) & 0 & 0 & 1 \\
$\phi(n,a_{1})$ & 0.0048515 & 0.0020649 & 1.0004691 \\ \hline
\end{tabular}
\end{center}

We can check more precisely how fast is the statistic $\phi (n,a_{1})$
converges in law to $\mathcal{N}$(0,1). We chose to compute the Kolmogorov
distance between $\phi (n,a_{1})$ and $\mathcal{N}(0,1)$. For this aim, we
approximate the cumulative distribution function using empirical cumulation
distribution function based on 500 replications of the computation of $\phi
(n,a_{1})$ for $n=3000$. The next figure shows the empirical and standard
normal cumulative distribution functions.
\begin{figure}[H]
\centering
\includegraphics[width=0.7\linewidth]{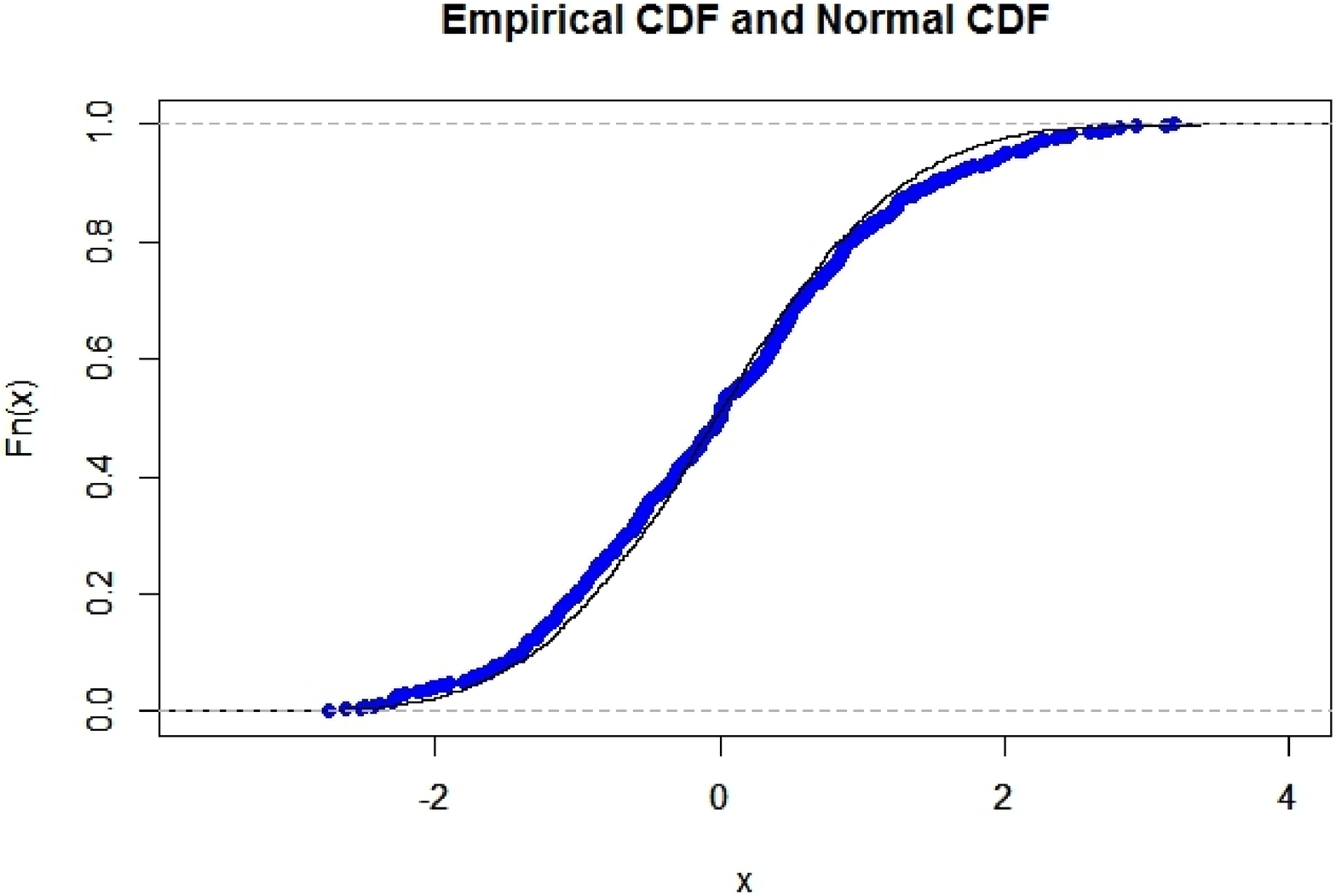} \label{fig:cdf}
\end{figure}
The Kolmogorov distance between the two laws, which equals the sup norm of
the difference of these cumulative distribution functions, computes to
approx. $0.052$. On the other hand, since (See for example Theorem 3.3 of
\cite{CGS} for a proof)
\begin{equation}
d_{Kol}(\phi (n,a_{1}),\mathcal{N}(0,1))\leqslant 2\sqrt{d_{W}(\phi
(n,a_{1}),\mathcal{N}(0,1))},  \label{dKoldW}
\end{equation}%
the distance on the left-hand side should be bounded above by $2\times
3000^{-1/4}=0.27$ approx times any constant coming from the upper bound in
Proposition \ref{GaussianEst}. This is five times larger than our estimate
of the actual Kolmogorov distance $0.052$, a reassuring practical
confirmation of Proposition \ref{GaussianEst}, and of our underlying results
on normal asymptotics of 2nd-chaos AR(1) quadratic variations. If that
proposition's upper bound with its rate $n^{-1/2}$ applied directly to the
Kolomogorov distance, as is known to be the case for the Berry-Esséen
theorem in the classical CLT, the value $0.052$ should be compared to $%
3000^{-1/2}=0.018$ approx., which is arguably in the same order of
magnitude. This is a motivation to investigate whether the so-called delta
method which we used here to prove Proposition \ref{GaussianEst} under the
Wasserstein distance, could also apply to the total variation distance,
since it is known to be an upper bound on the Kolmogorov distance without
the need for the square root as in the comparison (\ref{dKoldW}).

\end{document}